\documentclass[a4paper,10pt]{amsart}
\usepackage[utf8]{inputenc}
\usepackage{amsthm,amsmath,amssymb,mathtools,mathrsfs,setspace}
\usepackage{bm}
\usepackage{graphicx}
\usepackage{todonotes}
\usepackage{ esint }
\usepackage{thmtools} 
\allowdisplaybreaks
\usepackage[doi=false,isbn=false,url=false]{biblatex}

\usepackage[pdfdisplaydoctitle,colorlinks,breaklinks,urlcolor=blue,linkcolor=blue,citecolor=blue]{hyperref} 
\allowdisplaybreaks

\newcommand{\B}{\mathcal{B}}
\newcommand{\C}{\mathbb{C}}

\newcommand{\N}{\mathbb{N}}

\newcommand{\R}{\mathbb{R}}

\newcommand{\T}{\mathbb{T}}

\newcommand{\U}{\mathcal{U}}

\newcommand{\Z}{\mathbb{Z}}

	\renewcommand{\P}{\mathbb{P}}

\newcommand{\ovl}{\overline}
\newcommand{\ep}{\epsilon}

\newcommand{\cD}{\mathcal{D}}

\newcommand{\cL}{\mathcal{L}}
\newcommand{\cM}{\mathcal{M}}
\newcommand{\cN}{\mathcal{N}}

\newcommand{\cQ}{\mathcal{Q}}

\newcommand{\cW}{\mathcal{W}}

\DeclareMathOperator{\supp}{supp}

\let\div\relax
\DeclareMathOperator{\div}{div}

\renewcommand{\epsilon}{\varepsilon}

\newcommand{\brak}[1]{\left\langle#1\right\rangle}

\newtheorem{theorem}{Theorem}[section]
\newtheorem{definition}[theorem]{Definition}
\newtheorem{corollary}[theorem]{Corollary}
\newtheorem{lemma}[theorem]{Lemma}
\newtheorem{proposition}[theorem]{Proposition}

\theoremstyle{remark}
\newtheorem{rmk}[theorem]{Remark}
\newtheorem{Nota}[theorem]{Nota}

\numberwithin{equation}{section}

\usepackage{biblatex} 
\usepackage{caption}
\usepackage{subcaption}

\providecommand{\U}[1]{\protect\rule{.1in}{.1in}}
\newcommand{\abbr}[1]{{\tt{\uppercase{\scalebox{0.75}{#1}}}}}

\begin{document}

\title{Smoluchowski coagulation equation with velocity dependence}

\author[F. Flandoli]{Franco Flandoli}
\address{Scuola Normale Superiore, Piazza dei Cavalieri, 7, 56126 Pisa, Italia}
\email{\href{mailto:franco.flandoli@sns.it}{franco.flandoli@sns.it}}
\date\today

\author[R. Huang]{Ruojun Huang}
\address{Scuola Normale Superiore, Piazza dei Cavalieri, 7, 56126 Pisa, Italia}
\email{\href{mailto:ruojun.huang@sns.it}{ruojun.huang@sns.it}}

\author[A. Papini]{Andrea Papini}
\address{Scuola Normale Superiore, Piazza dei Cavalieri, 7, 56126 Pisa, Italia}
\email{\href{mailto:andrea.papini@sns.it}{andrea.papini@sns.it}}

\date\today
\maketitle

\begin{abstract}
In the present article, we introduce a variant of Smoluchowski's coagulation equation with both position and velocity variables taking a kinetic viewpoint arising as the scaling limit of a system of second-order (microscopic) coagulating particles. We focus on the rigorous study of the \abbr{PDE} system in the spatially-homogeneous case proving existence and uniqueness under different initial conditions in suitable weighted space, investigating also the regularity of such solutions.
\end{abstract}

\tableofcontents

\newpage
\section{Introduction}
Coagulation processes are ubiquitous in nature, from the movement of cells to the atmosphere, and in general complicated to understand experimentally and mathematically. In this work, we are motivated by the question whether a turbulent velocity field in the atmosphere enhances the coagulation of small rain droplets, and therefore favors rain fall. Physics literature on this topic has been vast, at least as early as Saffman-Turner \cite{SaTu} in the fifties, see also \cite{FFS, PW, Shaw}, arguing in favor of such coagulation enhancement. Here and in subsequent works, we give a novel approach to this problem that is fully mathematical. In particular, we take a kinetic viewpoint and study rigorously a variant of Smoluchowski's coagulation equation with velocity dependence that is akin to Boltzmann equation. It arises as the scaling limit of a system of second-order (microscopic) coagulating particles, modelling the interactions of rain droplets in the clouds, which are subjected to a common noise of transport type. Such a noise, constructed in recent mathematical works \cite{Galeati, FGL}, possesses several characteristics of real turbulence, such as it enhances diffusion of passive scalars. In the present work, we focus on the existence, uniqueness and regularity of this new \abbr{PDE}, after briefly introducing its origin.

Smoluchowski's classical equation \cite{smoluchowski1916} provides a first model for the time evolution of the probability distribution $\{f_m(t,x)\}_{m=1}^\infty$ of diffusing particles
of different sizes (or masses) $m\in\mathbb N$, say in $\T^d:=(\R/\Z)^d$, when they undergo pairwise coagulation with certain coagulation rate $\alpha(m,n)$:
\begin{align*}
    \partial_t f_m(t, x)=\Delta f_m(t,x)+&\sum_{n=1}^{m-1}\alpha(n,m-n)
    f_n(t,x)f_{m-n}(t,x)\\
    -&2    \sum_{n=1}^\infty \alpha(m,n)
    f_m(t,x)f_n(t,x), \quad t>0,\, x\in\T^d, \, m\in\mathbb N.
\end{align*}
The nonlinearity has two parts, a gain term and a loss term. Such a system of equations has been derived from scaling limits of Brownian particle systems by Hammond-Rezakhanlou \cite{HR, HR2}. To model the influence of a large-scale turbulent flow, it is natural to introduce a common noise. If we adopt a transport noise of the type
in \cite{FGL}
\begin{align*}
        \dot{\mathcal W}(t,x)    =\sum_{k\in K}\sigma_k(x)\dot{W}_t^k
\end{align*}
where $\{\sigma_k(x)\}_{k\in K}$ is a countable collection of divergence-free smooth vector fields and $\{W_t^k\}_{k\in K}$ independent one-dimensional Brownian motions, then we get a stochastic version of Smoluchowski's equation (an \abbr{SPDE})

\begin{equation}
\begin{aligned}\label{old-smol}
    d f_m(t,x)=&
    \Delta f_m(t, x)dt  +   \sum_{n=1}^{m-1}\alpha(n,m-n)
    f_n(t,x)f_{m-n}(t,x)dt    \\
    -& 2  \sum_{n=1}^\infty \alpha(m,n)    f_m(t,x)f_n(t,x)dt
    -\sum_{k\in K}\nabla f_m(t,x)\cdot\sigma_k(x)\, dW_t^k\\
    &+ \text{div}\left(\cQ(x,x)\nabla f_m(t,x)\right),\quad m\in\N
\end{aligned}
\end{equation}
where $\cQ(x,x):=\sum_{k\in K}\sigma_k(x)\otimes\sigma_k(x)$ coming from the transport-type noise. Under specific choice of $\{\sigma_k\}_{k\in K}$, we can have that $\cQ(x,x)\equiv \kappa I_d$, for an enhanced diffusion coefficient $\kappa>0$, the so-called ``eddy diffusion" \cite{Boussinesq}. This picture, in its special case of finitely many mass levels $m=1,2,...,M$
and unit coagulation rate $\alpha(m,n)\equiv 1$, has been derived from particle systems in \cite{FH}. Another version of the \abbr{SPDE} with continuous mass variables $m\in\R$ has also been derived from particle systems with mean-field interactions in \cite{AP}.

However, conceptually and mathematically, the most difficult step
in this program
is to verify
that diffusion enhancement leads to
coagulation enhancement, namely, the fast increase of probability densities $f_m$ for $m\gg 1$ (large masses) for large diffusion coefficient case. In fact, the model \eqref{old-smol} turns out to be too crude, and even numerically we cannot verify a coagulation enhancement.

The problem lies in the fact that quick diffusion of masses may not lead to enhanced collision unless the coagulation rate depends on the velocity variable. Otherwise, the masses merely move around. 
We introduce a new system with both position and velocity variables. In the atmospheric physics literature, e.g. \cite{FFS, PW, Shaw, GW}, it is also common to consider cloud particles coagulate with a rate that is proportional to (when $d=3$)
\begin{align*}
    \alpha(m_1,m_2):=|v_1-v_2|(r_1+r_2)^2
\end{align*}
where $v_i$, $i=1,2$ are the velocities of two colliding rain droplets and $r_i:=m_i^{1/3}$, $i=1,2$ their respective radius. Under certain simplifications, it leads to the following kinetic version of Smoluchowski's equation (cf. Appendix \ref{appen-particle})
\begin{align}\label{main-eq-intro}
\begin{cases}
\partial_t f_m (t,x,v)
= -v\cdot \nabla_xf_m+c(m)\text{div}_v\left(vf_m\right)+\kappa c(m)^2\Delta_vf_m + Q_m(f,f)\\[5pt]
f_m|_{t=0}=f^0_m(x,v),  \quad m=1,...,M,
\end{cases}
\end{align}
where $(t,x,v)\in[0,T]\times\T^d\times\R^d$, and
\begin{equation}\label{collison-term}
\begin{aligned}
Q_m(f,f)(t,x,v) &:= \sum_{n=1}^{m-1}\iint_{\{nw'+(m-n)w=mv\}} s(n,m-n)f_n(t,x,w')f_{m-n}(t,x,w)|w-w'|dwdw'\\
&\quad -2\sum_{n=1}^M \int s(n,m)f_m(t,x,v)f_n(t,x,w)|v-w|dw,
\end{aligned}
\end{equation}
where 
\begin{align}\label{mass-const}
c(m):=\alpha m^{(1-d)/d}, \quad s(n,m):=(n^{1/d}+m^{1/d})^{d-1}.
\end{align}
In Appendix \ref{appen-particle}, we sketch the proof of the scaling limit from a coagulating microscopic particle system subjected to a common noise, to an \abbr{SPDE} that eventually gives rise to this \abbr{PDE}. Although it is not fully rigorous, it should justify the interest of this equation. Here again, turbulence contributes to large $\kappa$ versus small $\kappa$ when no turbulence. The eddy diffusion occurs now in the velocity variable.

The aim of our research is two-fold. Theoretically, we are interested in proving the well-posedness of the \abbr{PDE} and associated \abbr{SPDE} cf. \eqref{smol-spde}, as well as the passage from the one to the other. Then, both theoretically and numerically we aim to demonstrate that the larger the $\kappa$, i.e. the more intense is the turbulence, the faster masses coagulate. See \cite{FHP} for a first numerical study in this direction. In the present article, we focus on the \abbr{PDE} system \eqref{main-eq-intro} in the spatially-homogeneous case, i.e. by considering the initial conditions $f_0^m$ constant in $x$ for every $m$, we can reduce \eqref{main-eq-intro} to 
\begin{align}\label{reduced-eq-intro}
\begin{cases}
\partial_t f_m (t,v)
= c(m)\text{div}_v\left(vf_m(t,v)\right)+\kappa c(m)^2\Delta_vf_m(t,v) + Q_m(f,f)(t,v)\\[5pt]
f_m|_{t=0}=f^0_m(v), \quad m=1,...,M,
\end{cases}
\end{align}
where $(t,v)\in[0,T]\times\R^d$ and $Q_m(f,f)$ is as in \eqref{collison-term} but without the $x$-dependence, and we
prove existence, uniqueness and regularity of the solutions of \eqref{reduced-eq-intro}, for every fixed $\kappa>0$. 

Denote a weighted $L^p$ space
\begin{align}\label{weighted-Lp}
    L^p_k(\R^d):=\left\{f:\R^d\rightarrow\R\ \text{ s.t. }\ f\brak{v}^k\in L^p(\R^d)\right\}, \quad p\in[1,\infty],\, k\in\N,
\end{align}
where 
\[
\brak{v}:=\sqrt{1+|v|^2}
\]
and a weighted Sobolev space
\begin{align}\label{weighted-sobolev}
    H^n_k(\R^d):=\left\{f\in L^2_k\ \text{ s.t. } \nabla^\ell f\in L^2_k(\R^d), \; \forall 1\leq \ell\leq n\right\}.
\end{align}
The main result of this article is as follows.
\begin{theorem}
Fix any finite $T$ and $\kappa>0$. 
Suppose that initial conditions $f_m^0(v)\in (L^2\cap L^1_2)(\R^d)$ and nonnegative, for every $m=1,...,M$, then there exists at least one nonnegative solution in the class $L^\infty\left([0,T]; L^1_2(\R^d)\right)^{\otimes M}$. 

If the initial conditions $f_m^0(v)\in (H^1_1\cap L^1_2)(\R^d)$ and nonnegative, then there exists a unique nonnegative solution, and in this case $f_m(t)\in C_b^\infty(\R^d)$ for any $t>0$.
\end{theorem}
The most difficult part in our opinion is uniqueness, due mainly to the presence of $|w-w'|$ in the nonlinear term, and the fact that the velocity variable $v\in\R^d$ is unbounded, hence in the presence of Laplacian, $f_m(t,v)$ is never compactly supported even if starting with so. These together with the fact that we have a system rather than just one equation, cause a severe difficulty in closing a Gronwall inequality for uniqueness. As far as we are able, the weighted $L^1$ space is the only one in which a Gronwall argument can work, even if one is willing to assume that solutions are Schwartz functions. (The problem is related to integrability rather than smoothness.) Indeed, with weighted $L^1$ we can find certain cancellations that remove those terms with higher weights brought by the kernel, and this seems not achievable with other spaces such as weighted $L^2$. Equally essential to this cancellation is considering the sum over the norms of all the densities $f_m$, $m=1,..,M$, rather than treating them individually. Indeed, this is already essential to derive various apriori estimates.

On the other hand, existence is proved by constructing a family of approximating equations each corresponding to a truncation of the kernel $|w-w'|$ in the nonlinearity. These approximating problems are more amenable to study since the difficulty related to the kernel is no longer severe, for each fixed truncation parameter.

There is an unexpected connection to the vast area of Boltzmann equations \cite{Villani}. Our \eqref{main-eq-intro} may be viewed as a Boltzmann-type equation with perfectly inelastic collision, rather than the classical elastic collision. Indeed, it is derived from particles undergoing pairwise coagulation, hence two particles merge into one based on the principle of conservation of momentum (and not energy). It is also local in nature in that the nonlinearity acts on the velocity variable, per $(t,x)$. The closest works in the Boltzmann literature seem to be the ones on excited granular media, see \cite{GPV} and references therein, and on multi-species Boltzmann equations, see \cite{BD} and references therein. In a sense our equation combines the features of both of them. From a technical point of view, the aforementioned difficulty with uniqueness are also present in \cite{MW, GPV} and some references therein, and we have learned from these sources. On the other hand, there are various differences that set our model apart from these references. Since $M<\infty$, we do not have the conservation of mass and momentum, and we do not expect nontrivial stationary solutions -- indeed all $f_m(t)$ should decay to zero as $t\to\infty$ (i.e. eventually all masses are transferred out of the system). In general, our nonliearity $Q_m(f,f)$ does not enjoy any particular kind of symmetry. Our derivation of the apriori estimates is also quite different, in particular we need not invoke entropy estimates and Povzner-type inequalities, as are standard in the Boltzmann literature.

Without loss of generality and simplifying notations, in the main part of the paper, we set $\kappa=1$, $c(m)=1$, $s(n,m)=1$.

\section{Deterministic Smoluchowski system with Velocity}\label{sec:apriori-est}
\subsection{Non-local Smoluchowski Equations}
Recall our system of nonlinear integral-differential equations for $\{f_m(t,x,v)\}_{m=1}^M$ defined as:
\begin{align}\label{main-eq}
\begin{cases}
\partial_t f_m = -v\cdot \nabla_xf_m+\text{div}_v\left(vf_m\right)+\Delta_vf_m + Q_m(f,f)\\[5pt]
f_m|_{t=0}=f^0_m
\end{cases}
,\ \quad m=1,...,M
\end{align}
where $(t,x,v)\in [0,T]\times \mathbb{T}^d\times \R^d$. The nonlinear term $Q_m(f,f)$, representing the masses interactions is given by \eqref{collison-term}.
Since we expect the system to represent density of particles with different masses we take $f^0_m\geq 0,\forall\  m=1,...,M$. 

We focus on the case when the initial condition is constant in the space variable $x$, and so this property pass to the solution of $(\ref{main-eq})$. We get a new set of equations solved by $f_m(t,v), m=1,...,M$, where we omit the subscript $v$ in the derivatives in the sequel (since there is no ambiguity)
\begin{align}\label{main-eq1}
\begin{cases}
\partial_t f_m = \text{div}\left(vf_m\right)+\Delta f_m + Q_m(f,f)\\[5pt]
f_m|_{t=0}=f^0_m
\end{cases}
, \quad m=1,...,M
\end{align}
where $(t,v)\in [0,T]\times \R^d$ and $f^0_m\geq 0,\ \forall\ m$.
By a change of variables
\[
w'=\varphi(v,w):=\frac{mv-(m-n)w}{n}, 
\]
the nonlinear term can be written as:
\begin{align*}
Q_m(f,f)(t,v) &:= \sum_{n=1}^{m-1}\iint_{\{nw'+(m-n)w=mv\}} f_n(t,w')f_{m-n}(t,w)|w-w'|dwdw'\\
&\quad -2\sum_{n=1}^M \int  f_m(t,v)f_n(t,w)|v-w|dw\\
&\ =\sum_{n=1}^{m-1}\int f_n\left(t,\varphi(v,w)\right)f_{m-n}(t,w)\left(\frac{m}{n}\right)^2|v-w|dw\\
&\quad -2\sum_{n=1}^M \int f_m(t,v)f_n(t,w)|v-w|dw.
\end{align*}
We will consider the following notion of weak solutions to \eqref{main-eq1}.
\begin{definition}[$L^\infty_t L_v^1$-weak solution]\label{def:weak-sol}
Fix $T>0$, and $f^0_m\in L^1_2\cap L^2$. A solution for equation $(\ref{main-eq1})$ is a set of functions $f_m\in L^\infty(0,T;L_2^1(\R^d))$ indexed by the masses $m=1,...,M$ such that $\forall m,\ Q_m(f,f)\in L^\infty(0,T;L^1(\R^d))$ and the following holds:
\begin{align*}
   \sum_{m=1}^M\brak{f_m(t),\phi}&-\sum_{m=1}^M\brak{f_m^0,\phi}=\int_0^t \sum_{m=1}^M\brak{f_m(s),\Delta\phi}ds\\
   &-\int_0^t\sum_{m=1}^M\brak{f_m(s),v\cdot\nabla\phi}ds
   + \int_0^t \sum_{m=1}^M\brak{Q_m(f(s),f(s)),\phi}ds,
\end{align*}
for a.e. $t\in [0,T]$, and $\forall\phi\in C^\infty_c(\R^d)$.
\end{definition}
\subsection{A priori estimate and conserved quantities for system $(\ref{main-eq1})$}
We prove for the system with finite number of masses a decay in time of the total mass and the conservation of such quantity in the infinite system.

Our first result is valid even in the system with space variable, hence we state it more generally.
\begin{lemma}\label{lem:conserved}
Suppose that $f_m^0(x,v)\in L^1(\T^d\times\R^d)$ and $\{f_m(t)\}_{m=1}^M$ is a nonnegative solution of the equation $(\ref{main-eq})$, such that $f_m(t)$ and its gradient $\nabla f_m(t)$ have a fast decay at infinity. Then if $M<\infty$, then the quantity 
\begin{align}\label{total_rain}
\mathcal T(t):=\iint \sum_{m=1}^M m\, f_m(t,x,v)dxdv
\end{align}
is non-increasing in $t$; if $M=\infty$, then it is constant in $t$.
\end{lemma}
\begin{rmk}\label{mass-leave}
It is not clear how to relate the finite- and infinite-level systems. We only work with $M<\infty$. Since $t\mapsto \mathcal T(t)$ is decreasing in this case, we interpret $\mathcal T(0)-\mathcal T(t)$ as the amount of mass falling out of the system. It can be viewed as an indicator of how efficient the coagulation is (by transferring densities from small masses to large masses, see \cite{FHP}).
\end{rmk}
\begin{proof}
Consider first the nonlinear part. For every fixed $t,x$ we have that 
\begin{align*}
&\int\sum_{m=1}^M m\, Q_m(f,f)dv\\
&= \sum_{m=1}^M \sum_{n=1}^{m-1}m\int dv \iint_{\{nw+(m-n)w'=mv\}} f_n(t,x,w)f_{m-n}(t,x,w')|w-w'|dwdw'\\
&\quad -2\sum_{m=1}^M\sum_{n=1}^M m\iint f_m(t,x,v)f_n(t,x,w)|v-w|dwdv\\
&=\sum_{m=1}^M \sum_{n=1}^{m-1}m\iint f_n(t,x,w)f_{m-n}(t,x,w')|w-w'|dwdw'\\
&\quad -2\sum_{m=1}^M\sum_{n=1}^M m\iint f_m(t,x,v)f_n(t,x,w)|v-w|dwdv
\end{align*}
Exchanging the order of summations in the first term, and setting $k=m-n$, the above equals
\begin{align*}
&\sum_{n=1}^M \sum_{k=1}^{M-n} (k+n)\iint f_n(t,x,w)f_k(t,x,w')|w-w'| dwdw'\\
&\quad -2\sum_{m=1}^M\sum_{n=1}^M m\iint f_m(t,x,v)f_n(t,x,w)|v-w|dwdv\\
&\le \sum_{n=1}^M \sum_{k=1}^{M} (k+n)\iint f_n(t,x,w)f_k(t,x,w')|w-w'| dwdw'\\
&\quad -2\sum_{m=1}^M\sum_{n=1}^M m\iint f_m(t,x,v)f_n(t,x,w)|v-w|dwdv=0
\end{align*}
where equality is achieved if and only if $M=\infty$. For the linear part, we note that for every $m, t,x$ we have that 
\begin{align*}
\int_{\R^d} \text{div}_v\left(vf_m+\nabla_v f_m\right)dv = 0
\end{align*}
due to integration by parts and $|vf_m|+|\nabla_vf_m|\to0$ as $|v|\to\infty$ for every fixed $t,x$. Also, for every $m,t,v$ we have that 
\begin{align*}
\int_{\R^d} v\cdot\nabla_xf_m dx = 0
\end{align*}
due to integration by parts and $|f_m|\to0$ as $|x|\to\infty$ for every fixed $t,v$. Thus we can conclude since $\mathcal T(t)$ \eqref{total_rain} can be written as
\begin{align*}
\partial_t\mathcal T &= \int_{\R^d}\int_{\R^d} \sum_{m=1}^M m \; \partial_tf_m(t,x,v)dxdv= -\sum_{m=1}^M m\iint v\cdot \nabla_x f_m dxdv \\
&+\sum_{m=1}^M m\iint  \text{div}_v\left(vf_m+\nabla_v f_m\right)dvdx + \iint \sum_{m=1}^M m\, Q_m(f,f)dvdx\le 0
\end{align*}
with equality if and only if $M=\infty$. 
\end{proof}
We specialize now to consider the spatially-homogeneous case, i.e. $f_m(t,x,v)=f_m(t,v)$ independent of $x\in\T^d$.

Recall the weighted $L^p$ spaces $L^p_k(\R^d)$ \eqref{weighted-Lp}
and we will indicate by $L^p_{k,M}$ the space in which the sequence $\{f_m(t)\}_{m=1}^M$ lives, with norm given by 
\begin{align*}
\|f\|_{p,M,k}:=\sum_{m=1}^M\|f_m\|_{p,k}=\sum_{m=1}^M\left(\int |f_m(v)|^p\brak{v}^{pk} dv\right)^{1/p}.
\end{align*}
The reason we are going to introduce such spaces is motivated not only to understand the regularity of the solution of the system in study, but also to deal with the terms $|v-w|$ in the coagulation kernel since, as we have already noted in the explicit formulation of the nonlinearity $Q_m(f,f)$.

There are a few ways to control the norm of such a quantity, and we prove here the main estimates that we apply in all the a priori estimates for the equation.
\begin{proposition}\label{sobolev}
For every $p\in[1,\infty]$ and $k\geq0$, there exists a constant $C$ depending only on $p,k, M$ and the dimension $d$, such that we have the following bound on the nonlinearity
$$
\|Q_m(g,f)\|_{p,k}\leq C\left(\|g\|_{p,M,k+1}\|f\|_{1,M,k+1}\right)
$$
and more general
$$
\|Q(g,f)\|_{p,M,k}\leq MC\left(\|g\|_{p,M,k+1}\|f\|_{1,M,k+1}\right).
$$
\end{proposition}
\begin{proof}
We prove this with a simple computation, using the duality of the norm for $L^p$ spaces. For simplicity of notation during the proof we omit whenever is possible the appendix $M$ on the norm. Fix $m\in \{1,...,M\}$ and consider $Q_m(g,f)$ for $(g_n)_{n=1,...,M}$ and $(f_n)_{n=1,...,M}$:
\begin{align*}
    \|Q_m(g,f)\|_{p,k}&=\sup_{\phi\in L^{p'},\ \|\phi\|_{p'}=1}\int Q_m(g,f)\brak{v}^{k}\phi(v)dv,
\end{align*}
where $1/p+1/p'=1$. We now work directly on each term of the summation in the nonlinearity both in the positive and negative part. For the negative part we have:
\begin{align*}
    \iint g_m(v)f_n(w)|v-w|\brak{v}^k\phi(v)dwdv&\leq \iint g_m(v)f_n(w)\brak{v}^{k+1}\brak{w}\phi(v)dwdv\\
    &=\|f_n||_{1,1}\int g_m(v)\brak{v}^{k+1}\phi(v)dv\\
    &\underbrace{\leq}_{\|\phi\|_{p'}=1} \|g_m\|_{p,k+1}\|f_n||_{1,1}\leq \|g\|_{p,k+1}\|f\|_{1,k+1}.
\end{align*}
And now for the positive:
\begin{align*}
    &\iint c_m^n g_n(h(v,w))f_{m-n}(w)|v-w|\brak{v}^k\phi(v)dwdv\\
    &\lesssim \iint g_n(h)f_{m-n}(w)\brak{v}^{k+1}\brak{w}\phi(v)dwdv\\
    &\underbrace{=}_{h(v,w)=z}\iint g_n(z)f_{m-n}(w)\brak{\tilde{h}(z,w)}^{k+1}\brak{w}\phi(\tilde{h})dwdv\\
    &\leq \iint g_n(z)f_{m-n}(w)(\brak{z}^{k+1}\brak{w}+\brak{w}^{k+1})\phi(\tilde{h})dwdv\\
    &=\iint g_n(z)f_{m-n}(w)\brak{z}^{k+1}\brak{w}\phi(\tilde{h})dwdv+\iint g_n(z)f_{m-n}(w)\brak{w}^{k+1}\phi(\tilde{h})dwdv\\
    &=\int f_{m-n}(w)\int g_n(z)\brak{z}^{k+1}\phi(\tilde{h})dv\brak{w}dv+\int f_{m-n}(w)\brak{w}^{k+1}\int g_n(z)\phi(\tilde{h})dvdw\\
    &\underbrace{\leq}_{\text{H\"older\ inequality}}\int f_{m-n}(w)\|\brak{w}\|g_n\|_{p,k+1}dv+\int f_{m-n}(w)\brak{w}^{k+1}\|g_n\|_{p}dw\\
    &\underbrace{\leq}_{\text{Translation\ invariance}} 2\|g_n\|_{p,k+1}\|f_{m-n}||_{1,k+1}\leq 2\|g\|_{p,k+1}\|f\|_{1,k+1}.
\end{align*}
Putting everything together we conclude the proof.
\end{proof}
We begin now proving the a priori estimates for the standard moments of regular enough solution to the system of coagulation equations.
\begin{lemma}\label{lem:moments}
For any $\ell\in\mathbb N$, suppose that $f_m^0\in L^1_\ell$ and $\{f_m(t)\}_{m=1}^M$ is a nonnegative solution of the equation $(\ref{main-eq1})$, such that $f_m(t)$ and its gradient $\nabla f_m(t)$ have a fast decay at infinity. Then, there exists some finite constant $C_\ell$ depending only on $\ell, d,$ and the initial data $\{f_m(0)\}_{m=1}^M$ such that for any $t\ge 0$, 
\begin{align}\label{eq:moment}
\sum_{m=1}^M\int  |v|^{\ell}f_m(t,v)dv\le C_\ell.
\end{align}
\end{lemma}
\begin{proof}
We first consider $\ell=2k$ even, and prove by induction on $k\in\N$. The case $k=0$ follows from  Lemma \ref{lem:conserved}. Now assume $k\ge 1$ and assume \eqref{eq:moment} is proved for the case $\ell=2(k-1)$. Note that for each $m=1,...,M$, by \eqref{main-eq1}
\begin{align*}
&\partial_t\int  |v|^{2k}f_m(t,v)dv=\int  |v|^{2k}\partial_tf_m(t,v)dv\\
=&\int |v|^{2k} \text{div}_v\left(vf_m\right)dv+\int |v|^{2k} \Delta_v f_mdv\\
& +\int dv \sum_{n=1}^{m-1}\iint_{\{nw+(m-n)w'=mv\}}|v|^{2k} f_n(t,w)f_{m-n}(t,w')|w-w'|dwdw'\\
&-2\int dv \sum_{n=1}^M \int |v|^{2k} f_m(t,v)f_n(t,w)|v-w|dw\\
= & -2k\int|v|^{2k} f_mdv+2k(2k+d-2)\int |v|^{2k-2}f_m dv \\
& + \sum_{n=1}^{m-1}\iint \left|\frac{nw+(m-n)w'}{m}\right|^{2k} f_n(t,w)f_{m-n}(t,w')|w-w'|dwdw'\\
&- 2\sum_{n=1}^M\iint |v|^{2k}  f_m(t,v)f_n(t,w)|v-w|dwdv.
\end{align*}
Above, we assumed that for every $t>0$,
\begin{align*}
|v|^{2k+1}|f_m|\to0, \quad |v|^{2k}|\nabla_vf_m|\to0,  \quad \text{ as }|v|\to\infty 
\end{align*}
such that the boundary terms vanish in the integration by parts. Summing the above in $m=1,...,M$ and noticing by Jensen's inequality 
\begin{align*}
 &\left|\frac{nw+(m-n)w'}{m}\right|^{2k}\le \left(\frac{n|w|+(m-n)|w'|}{m}\right)^{2k}\\
& \le \frac{n}{m}|w|^{2k}+\frac{m-n}{m}|w'|^{2k}\le |w|^{2k}+|w'|^{2k},
\end{align*}
whenever $1\le n\le m-1$, we obtain that 
\begin{align*}
&\partial_t\sum_{m=1}^M\int  |v|^{2k} f_m(t,v)dv\\
&\le  -2k\sum_{m=1}^M\int  |v|^{2k} f_m(t,v)dv+2k(2k+d-2)\sum_{m=1}^M\int |v|^{2k-2}f_m(t,v) dv\\
&+\sum_{m=1}^M \sum_{n=1}^{m-1}\iint \left(|w|^{2k}+|w'|^{2k}\right) f_n(t,w)f_{m-n}(t,w')|w-w'|dwdw'\\
&- 2\sum_{m=1}^M\sum_{n=1}^M\iint |v|^{2k} f_m(t,v)f_n(t,w)|v-w|dwdv\\
&\le -2k\sum_{m=1}^M\int  |v|^{2k} f_m(t,v)dv+2k(2k+d-2) C_{2(k-1)}\\
&+\sum_{n=1}^M \sum_{m=n+1}^M\iint \left(|w|^{2k}+|w'|^{2k}\right) f_n(t,w)f_{m-n}(t,w')|w-w'|dwdw'\\
&- 2\sum_{m=1}^M\sum_{n=1}^M\iint |v|^{2k} f_m(t,v)f_n(t,w)|v-w|dwdv,
\end{align*}
where we used the induction hypothesis that for some finite constant $C_{2(k-1)}$ independent of $t$, 
\begin{align*}
\sum_{m=1}^M\int |v|^{2k-2}f_m(t,v) dv\le C_{2(k-1)}.
\end{align*}
Further, setting $\ell=m-n$ we see that 
\begin{align*}
&\sum_{n=1}^M \sum_{m=n+1}^M\iint \left(|w|^{2k}+|w'|^{2k}\right) f_n(t,w)f_{m-n}(t,w')|w-w'|dwdw'\\
&- 2\sum_{m=1}^M\sum_{n=1}^M\iint |v|^{2k} f_m(t,v)f_n(t,w)|v-w|dwdv\\
&\le\sum_{n=1}^M \sum_{\ell=1}^{M-n}\iint \left(|w|^{2k}+|w'|^{2k}\right) f_n(t,w)f_{\ell}(t,w')|w-w'|dwdw'\\
&- 2\sum_{m=1}^M\sum_{n=1}^M\iint |v|^{2k} f_m(t,v)f_n(t,w)|v-w|dwdv\; \le 0.
\end{align*}
Thus we conclude that 
\begin{align*}
\partial_t\sum_{m=1}^M\int  |v|^{2k}f_m(t,v)dv\le &-2k \sum_{m=1}^M\int  |v|^{2k}f_m(t,v)dv+2k(2k+d-2)C_{2(k-1)}.
\end{align*}
Assuming the initial data is such that
\begin{align*}
A_0(k):=\sum_{m=1}^M\int |v|^{2k}f_m(0,v)dv
\end{align*}
is finite, it follows that for any $t\ge0$ we have that 
\begin{align*}
\sum_{m=1}^M\int  |v|^{2k}f_m(t,v)dv&\le A_0(k)e^{-2kt}+(2k+d-2) C_{k-1}(1-e^{-2kt})\\
&\le A_0(k)+(2k+d-2) C_{2(k-1)}=:C_{2k}.
\end{align*}
This completes the induction for $n=2k$ even. Finally, we turn to the case $n=2k-1$ for $k\ge 1$. Note that 
\begin{align*}
\sum_{m=1}^M\int |v|^{2k-1}f_m(v)dv&\le \sum_{m=1}^M\int_{|v|\le 1}|v|^{2k-1}f_m(v)dv+\sum_{m=1}^M\int_{|v|>1}|v|^{2k-1}f_m(v)dv\\
&\le \sum_{m=1}^M\int f_m(v)dv+\sum_{m=1}^M\int|v|^{2k}f_m(v)dv\le C_0+C_{2k}.
\end{align*}
\end{proof}

\begin{rmk}
Call $A_k:=\sum_{m=1}^M A_k^m:=\sum_{m=1}^M\int f_m\brak{v}^kdv$ and consider $\int Q(f,f)\brak{v}^k dv$. Although we have used other properties, this quantity can be also expanded in the following way:
\begin{align*}
    \int Q(f,f)\brak{v}^k dv\lesssim -2\sum_{m=1}^M\sum_{n=M-n+1}^M\iint f_m(v) f_n(w)|v-w|\brak{v}^k dwdv
\end{align*}
Using Proposition \ref{sobolev} and the modified triangular inequality: $-2|v-w|\brak{v}^k\leq -2\brak{v}^{k+1}+2|w|\brak{v}^k$, we get 
\begin{align*}
    \lesssim -2\sum_{m=1}^M\sum_{n=M-n+1}^M\iint f_m(v) f_n(w)|v-w|\brak{v}^k dwdv\leq 2 A_1 A_k-2A_{k+1}C_k^M
\end{align*}
where $C_k^M:=\inf_t \int f_M(w)dw\geq0$. This give us an equation for $A_k$ of the form:
$$
\frac{d}{dt}A_k\leq -2C_k^MA_{k+1}+ C_kA_k(A_0+A_1)+C_{k,k-1}A_{k-1})
$$
Using the already proved bound on $A_k$ and the aforementioned equation, integrating in time, is enough to show that moments of order $k
\geq2$ are controlled for $t_0>0$, however small, even if $f_0$ has not bounded moments of higher order.
\end{rmk}
We can prove also an a priori bound in $L^2_tH^1_v\cap L^\infty_tL^2_v$ of solutions of \eqref{main-eq1}, under regularity assumption.
\begin{lemma}\label{lem:H_1}
Suppose that $f_m^0\in L^1_2\cap L^2$ and $\{f_m(t)\}_{m=1}^M$ is a nonnegative solution of the equation $(\ref{main-eq1})$, such that $f_m(t)$ and its gradient $\nabla f_m(t)$ have a fast decay at infinity. There exists some constant $C>0$ depending on the initial data $\{f_m(0),\nabla f_m(0)\}_{m=1}^M$, $T<\infty$, $M<\infty$  such that 
\begin{align*}
\sup_{t\in[0,T]}\left( \sum_{m=1}^M\int  f^2_m(t,v)dv +\int_0^t\int |\nabla f_m(s,v)|^2dv\ ds\right)\le C.
\end{align*}
\end{lemma}
\begin{proof}
Note that for each $m=1,...,M$, by \eqref{main-eq}
\begin{align*}
&\frac{1}{2}\partial_t\int  f^2_m(t,v)dv=\int f_m(t,v)\partial_t f_m(t,v) dv\\
&= \int f_m \text{div}_v\left(vf_m\right)dv+\int f_m\Delta_v f_mdv\\
& +\int dv \sum_{n=1}^{m-1}\iint_{\{nw+(m-n)w'=mv\}} f_m(t,v)f_n(t,w)f_{m-n}(t,w')|w-w'|dwdw'\\
&-2\int dv \sum_{n=1}^M \int  f_m(t,v)f_m(t,v)f_n(t,w)|v-w|dw\\
&= \frac d2 \int f_m^2dv-\int |\nabla f_m|^2 dv \\
& + \sum_{n=1}^{m-1}\iint f_m\left(t,\frac{nw+(m-n)w'}{m}\right)f_n(t,w)f_{m-n}(t,w')|w-w'|dwdw'\\
&- 2\sum_{n=1}^M\iint f^2_m(t,v)f_n(t,w)|v-w|dwdv.
\end{align*}
Here, we assumed that 
$|v\nabla_vf_m|\to0$, 
as $|v|\to\infty$ for every $t$ such that the boundary terms vanish in the integration by parts. 

Further, we see that each term on the positive nonlinearity can be written using the Young's inequality has (same for $v$ or $w$ variable):
\begin{align*}
&\iint f_m\left(t,\frac{nv+(m-n)w}{m}\right)f_n(t,v)f_{m-n}(t,w)|v|dvdw=\\
&\le \int \left[\frac12 \left(\int f_m\Big(t,\frac{nv+(m-n)w}{m}\Big)f_n(t,v)|v|dv\right)^2+\frac12 f_{m-n}(t,w)^2\right]dw\\
&= \int \frac12 \left(\int f_m\left(t,\frac{nv+(m-n)w}{m}\right)f_n(t,v)|v|dv\right)^2dw+\frac12\|f_{m-n}(t,w)\|_{L^2}^2
\end{align*}
Thus we need to take care of the squared factor and we can do this using Jensen and normalizing respect the measure $(\int f_n(t,v)dv)^{-1}f_n(t,v)dv$, so that we obtain:
\begin{align*}
&\int \frac12 \left(\int f_m\left(t,\frac{nv+(m-n)w}{m}\right)f_n(t,v)|v|dv\right)^2dw\\
&\leq  \frac12 C_0\int\int f_m^2\left(t,\frac{nv+(m-n)w}{m}\right)|v|^2f_n(t,v)dvdw\\
&= (A)
\end{align*}
We have then, using translation invariance of Lebesgue measure:
\begin{align*}
&(A)=\frac12 C_0\int\left(\int f_m^2\left(t,\frac{nv+(m-n)w}{m}\right)dw\right)f_n(t,v)|v|^2dv\\
&\lesssim \|f_m\|^2_{L^2}\int |v|^2f_n(t,v)dv \leq C_1\|f_m\|^2_{L^2} 
\end{align*}
where the constant $C_1:=C_1(M)$ depend also on $M$, using the moment bound Lemma \ref{lem:moments} with $\ell=2$, since we assumed that $f_m^0\in L^1_2$.

We can repeat the same argument for the $|w|$ part. 
Thus we conclude that, erasing the negative term of the nonlinearity,
\begin{align*}
&\partial_t\sum_{m=1}^M\int  f^2_m(t,v)dv+\sum_{m=1}^M\int |\nabla f_m(t,v)|^2 dv\le  C(M,0,1,A)\sum_{m=1}^M\int  f^2_m(t,v)dv  
\end{align*}
where $C(M,0,1,A)>0$, and assuming the initial data is such that
\begin{align*}
C_1:=\sum_{m=1}^M\int |v|^2f_m(0,v)dv,\;  C_0:=\sum_{m=1}^M\int f_m(0,v)dv,\;  C_l:=\sum_{m=1}^M\int f^2_m(0,v)dv< \infty.
\end{align*}
It follows that for any $t\in [0,T]$ we pass to integral form
\begin{align*}
&\sum_{m=1}^M\int  f^2_m(t,v)dv+\kappa c(m)^2\int_0^t\ \sum_{m=1}^M\int |\nabla f_m(s,v)|^2 dv\ ds\\
&\le C_{l} + C(M,0,1,A)\int_0^t\sum_{m=1}^M\int  f^2_m(s,v)dvds\\
&\leq C_{l} + C(M,0,1,A)\int_0^t\left(\sum_{m=1}^M\int  f^2_m(s,v)dv + \int_0^s\ \sum_{m=1}^M\int |\nabla f_m(r,v)|^2 dv\ dr \right)ds
\end{align*}
From Gronwall's lemma, it follows:
\begin{align*}
&\sup_{t\in [0,T]}\left(\sum_{m=1}^M\int  f^2_m(t,v)dv+\int_0^t\ \sum_{m=1}^M\int |\nabla f_m(s,v)|^2 dv\ ds\right)\leq C_{l}\ e^{TC_{M,0,A,1}}\leq C(T)<\infty.
\end{align*}
\end{proof}
\begin{rmk}
Consider the constant in the previous lemma on the right hand side of the Gronwall inequality and consider $C_T:=\sup_{t\leq T}C_t<\infty$. Then from the lemma we can recover the following:
$$
\partial_t \left(e^{-C_Tt}\|f_t\|_{2,M}^2\right)+\int_0^te^{-C_Tt}\|f_s\|_{H^1,M}^2ds\leq 0
$$
using Gronwall' lemma and using the property of sup and exponential, we have that exists $C:=C_Te^{C_TT}<\infty$ depending only on the initial condition $f_0$ and not his gradient, such that:
$$
\sup_{t\leq T}\|f_t\|_{2,M}^2+\int_0^t\|f_s\|_{H^1,M}^2ds\leq C.
$$
\qed
\end{rmk}
Using the already proven a priori estimate and the bound on the nonlinearity we can also show an a priori bound on the solution in the weighted space $L^\infty_t(L^2_k)_v\cap L^2_t(H^1_{k})_v$, 
noting also that $|v|^k\le \langle v\rangle^k\le 2^{k-1}(1+|v|^k)$.

\begin{lemma}\label{lem:H_1-weighted}
For every $k\in\N$, suppose that $f^0_m\in L^1_{4k+2}\cap L^2_{k}$ and $\{f_m(t)\}_{m=1}^M$ is a nonnegative solution of the equation $(\ref{main-eq1})$, such that $f_m(t)$ and its gradient $\nabla f_m(t)$ have a fast decay at infinity. Then there exists some constant $C>0$ depending on the initial data $\{f_m(0),\nabla f_m(0)\}_{m=1}^M$, $T<\infty$, $M<\infty$ such that 
\begin{align*}
\sup_{t\in[0,T]}\left( \sum_{m=1}^M\int  f^2_m(t,v)|v|^{2k}dv + \int_0^t\int |\nabla f_m(s,v)|^2|v|^{2k}dv\ ds\right)\le C.
\end{align*}
\end{lemma}
\begin{proof}
We perform induction on $k\in\N$. The case $k=0$ is proved in Lemma \ref{lem:H_1}. 
Suppose that the thesis has been proved for the case $k-1$. Then, for each $m=1,...,M$, by \eqref{main-eq}
\begin{align*}
&\frac{1}{2}\partial_t\int  |v|^{2k}f^2_m(t,v)dv=\int |v|^{2k}f_m(t,v)\partial_t f_m(t,v) dv\\
&= \int |v|^{2k}f_m \text{div}_v\left(vf_m\right)dv+\int|v|^{2k} f_m\Delta_v f_mdv\\
& +\int dv \sum_{n=1}^{m-1}\iint_{\{nw+(m-n)w'=mv\}}|v|^{2k}f_m(t,v)f_n(t,w)f_{m-n}(t,w')|w-w'|dwdw'\\
&-2\int dv \sum_{n=1}^M \int  |v|^{2k}f_m(t,v)f_m(t,v)f_n(t,w)|v-w|dw\\
&\le (\frac{d}{2}-k)\int |v|^{2k}f_m^2dv-\frac{1}{2}\int |v|^{2k}|\nabla f_m|^2 dv +2 k^2\int |v|^{2k-2}f_m^2 dv\\
& + \sum_{n=1}^{m-1}\iint \left|\frac{nw+(m-n)w'}{m}\right|^{2k}f_m\left(t,\frac{nw+(m-n)w'}{m}\right)f_n(t,w)f_{m-n}(t,w')|w-w'|dwdw'\\
&- 2\sum_{n=1}^M\iint |v|^{2k}f^2_m(t,v)f_n(t,w)|v-w|dwdv,
\end{align*}
where we used Young's inequality
\begin{align*}
    -2k\int |v|^{2k-2} v\cdot\nabla f_m f_m dv &\le 2k\int |v|^{2k-1}|\nabla f_m|f_mdv\\
    &\le \frac{1}{2}\int |v|^{2k}|\nabla f_m|^2dv+2k^2\int |v|^{2k-2}f_m^2dv
\end{align*}
where we assumed that $|v|^{2k}|f_m\nabla f_m|\to0$, $|v|^{2k+1}|f_m|^2\to0$, 
as $|v|\to\infty$ for every $t$ such that the boundary terms vanish in the integration by parts. 

Further, since 
\begin{align*}
    \left|\frac{nw+(m-n)w'}{m}\right|^{2k}(|w|+|w'|)\le |w|^{2k+1}+|w'|^{2k+1}+|w|^{2k}|w'|+|w'|^{2k}|w|,
\end{align*}
we see that each term on the positive nonlinearity can be written using the Young's inequality has (same for $v$ or $w$ variable):
\begin{align*}
&\iint f_m\left(t,\frac{nv+(m-n)w}{m}\right)f_n(t,v)f_{m-n}(t,w)|v|^{2k+1}dvdw\\
&\le \int \left[\frac12 \Big(\int f_m\left(t,\frac{nv+(m-n)w}{m}\right)f_n(t,v)|v|^{2k+1}dv\right)^2+\frac12\left(f_{m-n}(t,w)\Big)^2\right]dw\\
&= \int \frac12 \left(\int f_m\left(t,\frac{nv+(m-n)w}{m}\right)f_n(t,v)|v|^{2k+1}dv\right)^2dw+\frac12\|f_{m-n}(t,w)\|_{L^2}^2
\end{align*}
Thus we need to take care of the squared factor and we can do this using Jensen and normalizing respect the measure $(\int f_n(t,v)dv)^{-1}f_n(t,v)dv$, so that we obtain:
\begin{align*}
&\frac12 \left(\int f_m\left(t,\frac{nv+(m-n)w}{m}\right)f_n(t,v)|v|^{2k+1}dv\right)^2dw\\
&\leq  \frac12 C_0\int\int f_m^2\left(t,\frac{nv+(m-n)w}{m}\right)|v|^{4k+2}f_n(t,v)dvdw\\
&= (A)
\end{align*}
We have then, using translation invariance of Lebesgue measure:
\begin{align*}
&(A)=\frac{1}{2}C_0\int\left(\int f_m^2\left(t,\frac{nv+(m-n)w}{m}\right)dw\right)f_n(t,v)|v|^{4k+2}dv\\
&\le C_0\|f_m\|^2_{L^2}\int |v|^{4k+2}f_n(t,v)dv \leq C_1\|f_m\|^2_{L^2} 
\end{align*}
where the constant $C_1:=C_1(M)$ depend also on $M$ and comes from the momentum bound of $f_m$, Lemma \ref{lem:moments} since we assumed that $f_m^0\in L^1_{4k+2}$.

Similarly, we have another term
\begin{align*}
    &\iint f_m\left(t,\frac{nv+(m-n)w}{m}\right)f_n(t,v)f_{m-n}(t,w)|v|^{2k}|w|dw\\
    &\le\int \frac12 \left(\int f_m\left(t,\frac{nv+(m-n)w}{m}\right)f_n(t,v)|v|^{2k}dv\right)^2dw +\int \frac12 f^2_{m-n}(t,w)|w|^2dw\\
    &\le C_0\int \frac12 \int f^2_m\left(t,\frac{nv+(m-n)w}{m}\right)f_n(t,v)|v|^{4k}dvdw +\int \frac12 f^2_{m-n}(t,w)|w|^2dw\\
    &\le C_0||f_m||_{L^2}^2\int f_n(t,v)|v|^{4k}dv+ \int \frac12 f^2_{m-n}(t,w)|w|^2dw\\
    &\le C_1||f_m||_{L^2}^2+\int \frac12 f^2_{m-n}(t,w)|w|^2dw.
\end{align*}
Thus we conclude that, erasing the negative term of the nonlinearity and summing over $n=1,...,m-1$, using the induction hypothesis,
\begin{align*}
&\partial_t\sum_{m=1}^M\int  |v|^{2k}f^2_m(t,v)dv+\frac{1}{2}\sum_{m=1}^M\int |v|^{2k}|\nabla f_m(t,v)|^2 dv\\
&\le  C(M,0,1,A)\sum_{m=1}^M\int  f^2_m(t,v)|v|^{2k}dv+ C_B,
\end{align*}
where $C(M,0,1,A),C_B>0$, and assuming the initial data is such that
\begin{align*}
C_1:=\sum_{m=1}^M\int |v|^{4k+2}f_m(0,v)dv, \; C_0:=\sum_{m=1}^M\int f_m(0,v)dv,\;  C_l:=\sum_{m=1}^M\int |v|^{2k} f^2_m(0,v) dv< \infty
\end{align*}
It follows that for any $t\in [0,T]$ we pass to integral form
\begin{align*}
&\sum_{m=1}^M\int  |v|^{2k}f^2_m(t,v)dv+\frac{1}{2}\int_0^t\ \sum_{m=1}^M\int |v|^{2k}|\nabla f_m(s,v)|^2 dv\ ds\\
&\le C_{l,B} + C(M,0,1,A)\int_0^t\sum_{m=1}^M\int  |v|^{2k}f^2_m(s,v)dvds\\
&\leq C_{l,B} + C(M,0,1,A)\int_0^t\left(\sum_{m=1}^M\int  |v|^{2k}f^2_m(s,v)dv + \int_0^s\ \sum_{m=1}^M\int |v|^{2k}|\nabla f_m(r,v)|^2 dv\ dr \right)ds
\end{align*}
From Gronwall's lemma, it follows:
\begin{align*}
\sup_{t\in [0,T]}&\left(\sum_{m=1}^M\int  |v|^{2k}f^2_m(t,v)dv+ \int_0^t\ \sum_{m=1}^M\int |v|^{2k}|\nabla f_m(s,v)|^2 dv\ ds\right)\\
&\leq C_{l,B}\ e^{TC_{M,0,A,1}}\leq C(T)<\infty.
\end{align*}
\end{proof}
We want to extend the inequality to higher derivatives and higher moments, to this end we consider the following proposition regarding the nonlinearity $Q$:

\begin{proposition}\label{ppn:lebeniz}
Let $f,g$ smooth and rapidly decaying function in $v$ at infinity, then
\begin{align*}
    \nabla Q_m(f,g)=Q_m(c_m^n\nabla f,g)+Q_m(f,c_m^n\nabla g),
\end{align*}
for each $m=1,..,M$.
\end{proposition}
\begin{proof}
First of all we split the nonlinearity into the positive and negative part and we change variables to get
\begin{align*}
    Q^+_m(f,g)&:=\sum_{n=1}^{m-1}\int \left(\frac{m}{n}\right)^2f_{n}(\varphi(v,w))g_{m-n}(w)|v-w| dw\\
    &=\sum_{n=1}^{m-1}\int  \left(\frac{m}{n}\right)^2f_{n}(\tilde{\varphi}(v,w))g_{m-n}(\theta(v,w))|w| dw\\
    Q^-_m(f,g)&:=\sum_{n=1}^M  f_m(v)\int g_n(w)|v-w|dw=\sum_{n=1}^M f_m(v)\int g_n(\theta(v,w))|w|dw,
\end{align*}
where we have set
$$
\varphi(v,w):=\frac{mv-(m-n)w}{n},\quad \theta(v,w):=v+w,\quad \tilde{\varphi}(v,w):=v-\frac{m-n}{n}w.
$$
Whit this we can use the differentiation under integral sign and since the dependence on $v$ is only on $f$ and $g$ we have
\begin{align*}
    &\nabla_v \left(\int f_{n}(\tilde{\varphi}(v,w))g_{m-n}(\theta(v,w))|w| dw\right)\\
    &=\int \left(\nabla_v f_{n}(\tilde{\varphi}(v,w))\right)g_{m-n}(\theta(v,w))|w| dw+\int f_{n}(\tilde{\varphi}(v,w))\nabla_v g_{m-n}(\theta(v,w))|w| dw.\\
    &\nabla_v  \left(f_m(v)\int g_n(w)|v-w|dw\right) = \nabla_v \left(f_m(v)\int g_n(\theta(v,w))|w|dw\right)=\\
    &=\nabla_v f_m(v)\int g_n(\theta(v,w))|w|dw+  f_m(v)\int \nabla_v g_n(\theta(v,w))|w|dw.
\end{align*}
Taking out the constant in the differentiation, depending only on $m$ and $n$, and summing all together we conclude the proof.
\end{proof}
\begin{rmk}
As a corollary, iterating Proposition \ref{ppn:lebeniz}, higher-order derivatives of $Q$ can be
calculated using the following formula:
\begin{align*}
    \partial^jQ(f,g)=\sum_{0\leq l\leq j} \binom{j}{l}Q_m\left(c_{m,n}\partial^{j-l} f,c_{m,n}\partial^l g\right)
\end{align*}
where $j$ is a multi-index $|j|=n$ such that $j=j_1...j_d$ (with indices possibly being zeroes) and
$$
\partial^j:=\partial_{v_1}^{j_1}...\partial^{j_d}_{v_d}.
$$
The factor $\binom{j}{l}$ is a multinomial coefficients and the sum is intended in the sense of ordering of multi-indices.
\end{rmk}
With this in mind we are able to state the following a priori estimate on the space $H^n_k$ for all $k,n\in \N$, defined in \eqref{weighted-sobolev} 
equipped with the norm $\|f\|_{H_k^n}:=\left(\sum_{0\le|j|\le n}\|\partial^j f\|_{2,k}^2\right)^{1/2}$.
\begin{lemma}\label{lem:H^n_k}
Suppose that $\{f_m\}_{m=1}^M$ is a nonnegative solution of the equation $(\ref{main-eq-trunc1})$, such that $f_m$ is regular enough. For all $n\in\N,\ k\in\N$, there exists some constant $C>0$ depending on the initial data $\{f_m(0)\}_{m=1}^M$, $T<\infty$, $M<\infty$, $n,k$ such that
\begin{align*}
\sup_{t\in[0,T]}\left( \sum_{m=1}^M \|f_m\|^2_{H^{n}_{k}}\right)\le C.
\end{align*}
\end{lemma}
\begin{proof}
We work by induction. We know form previous lemma that $n=0,1$ we have already proven the a priori bound. Now we suppose that for every $p<n$ it's true that $f\in H^p_k$ for all $k$, and we prove that $f\in H^n_k$ for all $k$. To this end, we consider a multi index $j, |j|=n$ and the quantity $\partial^j f_m$. Using the equation for $f_m$ we get
\begin{align*}
    \partial_t \left(\partial^j f_m\right)=\Delta\left(\partial^j f_m\right)+ \div\left(v\partial^j f_m\right)+C_j\left(\partial^j f_m\right)+\sum_{0\leq l\leq j} \binom{j}{l}Q_m\left(\partial^{j-l} f,\partial^l f\right)
\end{align*}
where $C_j$ is either $0$ or $1$ depending if the multi index has the $i$-th component of the derivative of $\sum_{i=1}^d \partial^j\left(\partial_i f_m v_i\right)$. We now consider the quantity
$$
\sum_{m=1}^M\int |\partial^j f_m|^2\brak{v}^{2k} dv
$$
and derivate in time we have for each $m$:
\begin{align*}
\frac{1}{2}\partial_t\int  &\brak{v}^{2k}|\partial^j f_m|^2(t,v)dv=\int \brak{v}^{2k}\partial^jf_m(t,v)\partial_t \partial^jf_m(t,v) dv\\
&= \int \brak{v}^{2k}\partial^jf_m \text{div}\left(v\partial^jf_m\right)dv+\int\brak{v}^{2k} \partial^jf_m\Delta (\partial^jf_m) dv+\\
&+ \int C_j\brak{v}^{2k}|\partial^jf_m|^2dv
 +\sum_{0\leq l\leq j}\int Q_m(\partial^{j-l}f,\partial^lf)\partial^j f_m \brak{v}^{2k} dv\\
&\leq \frac d2 \int \brak{v}^{2k}|\partial^jf_m|^2 dv- k\int \brak{v}^{2k-2}|\partial^jf_m|^2 dv +\\
&+k \int \brak{v}^{2k}|\partial^jf_m|^2 dv- k\int \brak{v}^{k-2}|v||\partial^jf_m|^2 dv+\\
&-\int |\nabla (\partial^j f_m)|^2 \brak{v}^{2k}dv+\sum_{0\leq l\leq j}\int Q_m(\partial^{j-l}f,\partial^lf)\partial^j f_m \brak{v}^{2k} dv\\
&\leq C'_{d,k}\int \brak{v}^{2k}|\partial^jf_m|^2 dv -\int |\nabla (\partial^j f_m)|^2 \brak{v}^{2k}dv+\\
&+\underbrace{\sum_{0\leq l\leq j}\int Q_m(\partial^{j-l}f,\partial^lf)\partial^j f_m \brak{v}^{2k} dv}_{A}.
\end{align*}
We miss to understand now the nonlinearity. To this end we consider:
\begin{align*}
    \| Q_m(f,g)\|_2\leq  C\left(\sum_{n=1}^{m-1}\|f_{m-n}\|_{2,k+1}\|g_n\|_{1,k+1}+\|f_m\|_{1,k+1}\sum_{n=1}^M\|g_n\|_{2,k+1}\right).
\end{align*}
In particular we have:
\begin{align*}
    &\sup_{\|\phi\|_2=1}\iint C_{m,n}f_n(h(v,w))f_{m-n}(w)|v-w|\phi(v)dvdw\\
    &\leq \sup_{\|\phi\|_2=1}\iint C'_{m,n}f_n(z))f_{m-n}(w)|z-w|\phi(h'(z,w))dzdw\\
    &\lesssim\begin{cases}
        &\sup_{\|\phi\|_2=1}\int \int f_{m-n}(w)\phi(h'(z,w))\brak{w}dwf_n(z))\brak{z}dz\leq C_{n,m}\|f_n\|_{2,1}\|f_{m-n}\|_{1,1}\\[5pt]
        &\sup_{\|\phi\|_2=1}\int \int f_n(z))\phi(h'(z,w))\brak{z}dzf_{m-n}(w)\brak{w}dw\leq C_{n,m}\|f_{m-n}\|_{2,1}\|f_{n}\|_{1,1}.
    \end{cases}
\end{align*}
With this, in particular, we can select where to put the norm and the weight depending on the index of the derivative and we get summation of quantity that depends on lower order in $H^p_k$ with $p<n$ and a quantity depending on $\partial^j$.
\begin{align*}
  \sum_{0\leq l\leq j}\int &Q_m(\partial^{j-l}f,\partial^lf)\partial^j f_m \brak{v}^{2k} dv\leq \int Q_m(\partial^{j}f,f)\partial^j f_m \brak{v}^{2k} dv \\
  &+\int Q_m(f,\partial^j f)\partial^j f_m \brak{v}^{2k} dv  + \sum_{1\leq l\leq j-1}\int Q_m(\partial^{j-l}f,\partial^lf)\partial^j f_m \brak{v}^{2k} dv\\
  &\leq C_{m,n}(k,M)\sum_{m=1}^M\|\partial^j f_m\|^2_{2,k+\mu}.
\end{align*}
Now we note that:
\begin{align*}
    \|\partial^j f_m\|^2_{2,k+\mu}&\leq \delta \|\nabla\partial^j f_m\|^2_{2}+ c_\delta\|\partial^{j-1} f_m\|^2_{2,2(k+\mu)}\\
    &\leq  \delta \|\nabla\partial^j f_m\|^2_{2,2k}+ c_\delta\|\partial^{j-1} f_m\|^2_{2,2(k+\mu)}\\
    &\leq \delta \|\nabla\partial^j f_m\|^2_{2,2k} + C_{\delta, M, n-1}.
\end{align*}
We put together all the estimate and we get:
\begin{align*}
\frac{1}{2}&\partial_t\sum_{m=1}^M\int  \brak{v}^{2k}|\partial^j f_m|^2(t,v)dv-C'_{M,d,k}\sum_{m=1}^M\int \brak{v}^{2k}|\partial^jf_m|^2 dv \leq\\
&-\sum_{m=1}^M\int |\nabla (\partial^j f_m)|^2 \brak{v}^{2k}dv+\delta C(k,M)\sum_{m=1}^M\int |\nabla (\partial^j f_m)|^2 \brak{v}^{2k}dv+ C_M(n-1,k)
\end{align*}
Taking $\delta$ small enough, up to constant $C_\delta(M,k,n-1)$ we obtain
\begin{align*}
\frac{1}{2}\partial_t\sum_{m=1}^M\int  \brak{v}^{2k}|\partial^j f_m|^2(t,v)dv&-C'_{M,d,k}\sum_{m=1}^M\int \brak{v}^{2k}|\partial^jf_m|^2 dv \leq C_\delta(M,k,n-1).
\end{align*}
Thanks to Gronwall's lemma we conclude the proof.
\end{proof}
With all this a priori estimates follow also estimates on the $L^2_k$ norm of the solution. At the level of the a priori estimate, Lemmas \ref{lem:H^n_k} and \ref{lem:H_1} guarantee us that we have a bound for every $T>0$ on the derivatives in $L^2([0,T]\times\R^d)$ implying that after a arbitrary short time the derivatives $\partial^n f_m(t)\in L^2(\R^d)$ for any $n$ and thus they propagate in time. On the level of a priori estimates the solution results to be immediately infinitely smooth in $v$ and decay faster than any negative power of $|v|$ at infinity.\\
Using the a priori regularity in $H^n_k$ and the aforementioned propagation, we note that for every $t>0$ and $m=1,...,M$, $f_m(t)$ is in fact a Schwartz function  for $\forall t$ and this decay is enough to get that $\partial_t f\in L^2([0,T]\times\R^d)$. 
\section{Approximating Problems and preliminary}\label{sec:approx-plm}
We define a sequence of approximating problems, indexed by $R>0$, in the following way:
\begin{align}\label{main-eq-trunc1}
\begin{cases}\partial_t f^R_m = \text{div}\left(vf^R_m\right)+\Delta f^R_m + Q^R_m(f^R,f^R)\\[5pt]
f^R_m|_{t=0}=f^0_m
\end{cases}
, \quad m=1,...,M
\end{align}
where $(t,v)\in [0,T]\times\R^d$ and $f^0_m$, $\forall m=1,...,M$ are the same as $(\ref{main-eq1})$. The approximated nonlinear term is expressed as follows:
\begin{align*}
Q^R_m(f,f)(t,v) &:= \sum_{n=1}^{m-1}\chi_R(v)\iint_{\{nw+(m-n)w'=mv\}} f_n(t,w)f_{m-n}(t,w')|w-w'|\chi_R(w) dwdw'\\
&\quad -2\sum_{n=1}^M \chi_R(v)\int  f_m(t,v)f_n(t,w)|v-w|\chi_R(w)dw\\
&= \sum_{n=1}^{m-1}\chi_R(v)\int f_n\Big(t,\varphi(v,w)\Big)f_{m-n}(t,w)|v-w|\big(\frac mn \big)^2\chi_R(w) dw\\
&\quad -2\sum_{n=1}^M \chi_R(v)\int f_m(t,v)f_n(t,w)|v-w|\chi_R(w)dw
\end{align*}
where $\chi_R(z):=\chi_{\mathbb{B}(0,R)}(z)$ is the indicator function on the ball of radius $R>0$ centered at the origin in $\R^d$.
\subsection{A priori estimate for the approximating equations}
We start with the same a priori estimate for the moments and energy, under the same regularity assumption of $(\ref{main-eq1})$, that we'll verify once existence and uniqueness is shown.\\
We suppose that $f^0_m\geq0,\ \forall\ m=1,..,M$, the solution of $(\ref{main-eq-trunc1})$ is nonnegative and  regular enough and has a fast decay at infinity, then we have for the nonlinear term:
\begin{align*}
&\chi_R(v)\int f_n\Big(t,\varphi(v,w)\Big)f_{m-n}(t,w)|v-w|\chi_R(w) dw\\
&\leq\int f_n\Big(t,\varphi(v,w)\Big)f_{m-n}(t,w)|v-w| dw
\end{align*}
And thus we can recover the same a priori estimates, as in Lemmas \ref{lem:moments}, \ref{lem:H_1}, \ref{lem:H_1-weighted}, \ref{lem:H^n_k}, on the total mass, moments and energy in both $H^1$ and $H^n_k$ for the approximated problems with the same constant independent of $R$.
\begin{lemma}
For any $\ell\in\mathbb N\cup\{0\}$, there exists some finite constant $C_\ell$ depending only on $n,d$ and the initial data $\{f_m(0)\}_{m=1}^M$ such that
\begin{align*}
\sum_{m=1}^M\int  |v|^{\ell}f^R_m(t,v)dv\le C_\ell, \quad \forall R,
\end{align*}
as soon as $\{f_m^R\}_m\in L^2.$
\end{lemma}

\begin{lemma}\label{lem:H_1-trunc}
As soon as $\{f_m^R\}_m$ is regular enough, there exists some constant $C>0$ depending on the initial data $\{f_m(0),\nabla f_m(0)\}_{m=1}^M$, $T<\infty$, $M<\infty$ such that 
\begin{align*}
\sup_{t\in[0,T]}\left( \sum_{m=1}^M\int  |f^{R}_m(t,v)|^2dv +\int_0^t\int |\nabla f^R_m(s,v)|^2dv\ ds\right)\le C.
\end{align*}
The constant $C$ does not depend on R and  is the same as the one in the full problem.
\end{lemma}

\begin{lemma}
As soon as $\{f_m^R\}_m$ is regular enough, $\forall k\geq 0$, there exists some constant $C>0$ depending on the initial data $\{f_m(0),\nabla f_m(0)\}_{m=1}^M$, $T<\infty$, $M<\infty$ such that
\begin{align*}
\sup_{t\in[0,T]}\left( \sum_{m=1}^M\int  
|f^{R}_m(t,v)|^2|v|^{2k}dv +\int_0^t\int |\nabla f^R_m(s,v)|^2|v|^{2k}dv\ ds\right)\le C.
\end{align*}
The constant $C$ does not depend on R and  is the same as the one in the full problem.
\end{lemma}
\begin{lemma}
Suppose that $\{f^R_m\}_{m=1}^M$ is a nonnegative solution of the equation $(\ref{main-eq-trunc1})$, such that $f^R_m$ is regular enough. For all $n\in\N,\ k\in\N$, there exists some constant $C>0$ depending on the initial data $\{f_m(0)\}_{m=1}^M$, $T<\infty$, $M<\infty$, $n,k$  such that 
\begin{align*}
\sup_{t\in[0,T]}\left( \sum_{m=1}^M \|f^R_m\|^2_{H^{n}_{k}}\right)\le C.
\end{align*}
The constant C does not depend on R and is the same as the one in the full problem.
\end{lemma}
\subsection{Preliminary on semigroups}
We want to construct mild solution for our approximated problems. With this in mind,  we define the following operator on function defined on $\R^d$:
\begin{align*}
\mathcal{L}f:= \Delta f+  v\cdot\nabla f,
\end{align*}
and rewrite our system of equations as:
\begin{align*}
\begin{cases}\partial_t f^R_m -\mathcal{L} f_m^R= d f^R_m + Q^R_m(f^R,f^R)\\
f^R_m|_{t=0}=f^0_m
\end{cases}
, \quad m=1,...,M
\end{align*}
with the usual definition of space, time and initial condition.

We know, \cite{MPV}
, that endowed with its maximal domain
$$
D_{p,max}(\mathcal{L}):=\left\{u\in L^p(\R^d)\cap W^{2,p}_{loc}(\R^d):\ \mathcal{L}u\in L^p(\R^d)\right\}
$$
the operator $\mathcal{L}$ is the generator of a strongly continuous semigroup $(P_t)_{t\geq0}$ in $L^p(\R^d)$. In particular we can characterize the domain of the operator as follow:
\begin{lemma}[{\cite[Theorem 1]{MPV}}]
The domain $D_{p,max}(\mathcal{L})$ of the generator of the semigroup $(P_t)_{t\geq0}$ coincides with:
$$
D_p(\mathcal{L}):=\left\{u\in W^{2,p}(\R^d):\ v\cdot \nabla u\in L^p(\R^d)\right\}.
$$
\end{lemma}
Fur such a semigroup $(P_t)_{t\ge 0}$, one can derive the following properties \cite[Theorem 3.3]{MPV}.
\begin{proposition}
The operator $(\mathcal{L}, D_{p,max}(\mathcal{L}))$ generates a semigroup $(P_t)_{t\geq0}$ in $L^p(\R^d)$ which satisfies the estimate
$$
\|P_tf\|_p\leq \|f\|_p
$$
for every $f\in L^p(\R^d)$.
\end{proposition}
\begin{proposition}\label{ppn:grad-estim}
For every $f\in L^p(\R^d)$ and $T>0$, the function $P(\cdot)f$ belongs to $C((0,T],W^{2,p}(\R^d))\cap C^1((0,T],L^{p}_{\text{loc}}(\R^d))$, and satisfies the estimates for $t\in(0,T]$,
\begin{align*}
    \|D^2P_tf\|_p\leq \frac{C_T}{t}\|f\|_p,\ \|\nabla P_tf\|_p\leq\frac{C_T}{\sqrt{t}}\|f\|_p.
\end{align*}
\end{proposition}
\begin{proposition}\label{ppn:reg-semigp}
Let $T>0$ and $g\in C([0,T],L^p(\R^d))$ be given, and consider the mild solution $u$ of the
Cauchy problem 
\begin{align*}
    \begin{cases}
    \partial_t u-\mathcal{L} u = g\ &in\ [0,T]\times\R^d\\
    u(0) = f\ &in\ \R^d,
    \end{cases}
\end{align*}
with $f=0$. Then, $u$ belongs to $C([0,T],W^{2,p}(\R^d))\cap W^{1,p}_{\text{loc}}([0,T]\times\R^d)$.
\end{proposition}
Thus, we can define the notion of mild solution for our system. For simplicity of notation we omit the dependence on $m$.
\begin{definition}[$L_v^2$-mild solution for approximated problems]
We rewrite first our approximated problems as an evolution equation in $L^p$ space
\begin{align*}
    \Dot{f}+\mathcal{L}f=df+Q(f,f),\ t>0\,\ f(0)=f^0.
\end{align*}
Fix $T>0$, a function $f\in C([0,T],L^2)$ is said to be a $L^2$-mild solution to $(\ref{main-eq-trunc1})$ on $[0,T]$ if $f$ solves the integral equation
\begin{align*}
    f(t)=P_tf^0+\int_0^tP_{t-s}\left(df(s)+Q(f(s),f(s))\right)ds,\ t\in[0,T].
\end{align*}
\end{definition}
\subsection{Aubin-Lions' Theorem in full Space}
\begin{lemma}[Aubin-Lions weighted] \label{lem:aubin}
Consider $H^1(\R^d)$, the usual Sobolev space, and $L^2(\R^d,|v|^2dx)$, the weighted Lebesgue space containing functions $f$ for which $\int|f|^2|v|^2dv<\infty$. Then 
$$
H^1(\R^d)\cap L^2(\R^d,|v|^2dv)\subset \joinrel\subset L^2(\R^d)
$$
and the embedding is compact.
As a consequence for every $q,k\geq1$ we have that
$$
L^p\left(0,T;H^1(\R^d)\cap L^2(\R^d,|v|^2dv)\right)\cap W^{1,q}(0,T;H^{-k}(\R^d))\subset\joinrel\subset L^p(0,T;L^2(\R^d))
$$
 for $p\in [1,\infty]$ and the embedding is compact.
\end{lemma}
\begin{proof}
Call $X:=H^1(\R^d)\cap L^2(\R^d,|v|^2dv)$ and consider $Y\subset X$ a bounded set.\\
We claim that $\forall \varepsilon>0,\exists N>0$ such that for all $f\in Y$:
$$
\int_{B^c_N} f^2<\varepsilon/2,
$$
Where $B^c_N$ is the complement in $\R^d$ of the ball of radius $N$ and centered at the origin.
Otherwise, assume this is not the case, then for some $\varepsilon>0$, for every $N>0$ there is some $f\in Y$ such that $\|f\|_{L^2(B^c_N)}\geq\varepsilon$.\\ 
This implies that $\|f\|_{X}\geq \|f^2|v|^2\|_{L^1(B^c_N)}\geq N^2\varepsilon$ so that the set $Y$ is not bounded, and this contradicts the hypotheses. So we consider the ball $B_N$ and we can apply Rellich-Kondrachov theorem to ensure the compact embedding to $L^2(B_N)$.\\
Now we recall that compactness of a set $A\subset X$ means that $\forall\varepsilon>0$ we find $f_1,...,f_k\in X$, where $k=k(\varepsilon)$, such that $A\subset\cup_{i=1}^k B(f_i,\varepsilon)$.\\
This means that $\forall \varepsilon>0$ we find functions $f_1,...,f_k\in L^2(B_N),\ k=k(\varepsilon)$, such that $\forall f\in Y$ we have some $f_j$ such that 
$$\|f-f_j\|_{L^2(B_N)}<\varepsilon/2$$
We recall then that $\|f\|_{L^2(B^c_N)}<\varepsilon/2$, so we even have
$\|f-f_j\|_{L^2(\R^d)}<\varepsilon$,
which implies the compact embedding $Y\subset\joinrel\subset L^2(\R^d)$.\\
The second part of the theorem now follows from the usual Aubin-Lions Theorem cf. \cite[Corollary 5]{Simon}, thus concluding the proof.
\end{proof}

\section{Existence and Uniqueness for the approximating problems}\label{sec:exist-uniq}
\begin{theorem}[Existence and Uniqueness]\label{thm:approx-plm}
Consider $d\geq 1$. Given any initial value $f^0_m\in L^2,\forall
 m=1,...,M$, problem $(\ref{main-eq-trunc1})$ possesses a unique maximal $L^2$-mild solution $f:= f(\cdot; f^0)$ on $[0,T(f^0))$. The maximal interval is such that $[0,T(f^0))$ is an open interval in $\R_+$. In addition,  
 $$
 f\in L^\infty([0,T(f^0)), L^2(\R^d)^{\otimes M})\cap W_{loc}^{1,2}((0,T(f^0)), H^2(\R^d)^{\otimes M}).
 $$
In addition: if $T(f^0)<\infty$, then
$$
\sup_{T(f^0)/2<t<T(f^0)}\|f_t\|_2=\infty.
$$
\end{theorem}
\begin{proof}
Let $T_0>0$ be arbitrary and define $X_T:=L^\infty([0,T],L_M^2)$ for $T\in(0,T_0]$, where we endowed the space with the usual sup norm in time and $L_M^2$ norm in velocity as $\|f\|_{L^2_M}:=\sum_{m=1}^M\|f_m\|_{L^2}.$\\
Consider $f^0\in L_M^2$ then from Proposition \ref{ppn:grad-estim} we know that $P_tf^0\in X_T$. From the same theorem it is implied that exist a constant $\kappa:=\kappa(T_0)>0$ such that $\|P_t f^=\|\leq e^\kappa(T_0)t$; we show that this is enough to define a map $\Gamma$ from $X_T\rightarrow X_T$, given by
$$
\Gamma(f):=P_\cdot f^0+\int_0^\cdot P_{\cdot-s}(df-Q^R(f,f))ds,\ f\in X_T.
$$
First, thanks to the estimate on the semigroup and the property of the truncated nonlinearity, we have that, as soon as $f\in X_T$
\begin{align*}
\begin{cases}
    &\|Q^R_m(f,f)\|_2^2\leq C_R (\sum_{m=1}^M \|f_m\|_2)^2<\infty.\\[5pt]
    & \int_0^\cdot P_{\cdot-s}(df-Q^R(f,f))ds\in X_T.
\end{cases}
\end{align*}
In particular we obtain the first estimate from the following computation on each of the terms of the nonlinearity
$$
\int \left|\int |v-w|_Rf_{n}(h(v,w))f_{m-n}(w)dw\right|^2dv\leq 2R\int_{B_R} \|f_n\|_2^2\|f_{m-n}\|_2^2dv= 2R|B_R|\|f_n\|_2^2\|f_{m-n}\|_2^2.
$$
Now we notice that the function $\Gamma$ also satisfies
\begin{align*}
    \|\Gamma(f)-P_\cdot f^0\|_{X_T}&=\left\|\int_0^\cdot P_{\cdot-s}(df-Q^R(f,f))ds\right\|_{X_T}\\
    &\leq \kappa T\left(d+C_R(\sum_{m=1}^M \|f_m\|_2)\right)(\sum_{m=1}^M \|f_m\|_2),
\end{align*}
for all $f\in X_T$. And we can prove the continuity of the map
\begin{align*}
    \|\Gamma(f)-\Gamma(g)\|_{X_T}&=\left\|\int_0^\cdot P_{\cdot-s}(d(f-g)+\left(Q^R(g,g)-Q^R(f,f)\right)ds\right\|_{X_T}\\
    &\leq T\left(d+C_R\|f\|_{X_T}+C_R\|g\|_{X_T}\right)\|f-g\|_{X_T},
\end{align*}
for all $f,g\in X_T$.\\
If we find a way to restrict the mapping $\Gamma$ to some closed subset, and prove that it is a contraction then, with a standard fix point argument, we conclude local existence and uniqueness for each approximating problems.\\
Consider to this extent the norm of the initial condition $\|P_\cdot f^0\|_{X_T}=\gamma$, we choose the  ball in $X_T$, call it $B_T$, centered in $P_\cdot f^0$ with radius $\gamma$. So that every function $f\in B_T$ as norm $\|f\|_{X_T}\leq 2\gamma$. Then selecting 
$$
T<\min\left\{ \frac{1}{d+4C_R\gamma},\frac{1}{\kappa(d+C_R 2\gamma)}\right\}
$$
The two inequality implies that the map $\Gamma$ send function in $B_T$ to itself and it is a contraction. Therefore there exist a unique $\tilde{f}\in B_T$ such that $\Gamma(\tilde{f})=\tilde{f}$, that is,
$$
\tilde{f}\in L^\infty([0,T],L_M^2)\ \text{ and }\ \tilde{f}=P_{\cdot}f^0+\int_0^{\cdot}P_{\cdot-s}\left(d\tilde{f}+Q^R(\tilde{f},\tilde{f})\right)ds.
$$
and $\tilde{f}$ is a mild-$L^2$ solution to the approximated problems.\\

Clearly we can extend the solution $\tilde{f}$ to a unique maximal solution $\tilde{f}:=\tilde{f}(f^0)$ with maximal interval $[0,T(f^0))$ that must be open in $\R^+$.\\
Consider $T^0:=T(f^0)<\infty$ and suppose there exist an increasing sequence $t_i\rightarrow T^0$ such that $\|f_{t_i}\|_p\leq r<\infty,\ \forall i\in\N$. Fix now $\overline{T}>T^0$ and fix $\gamma(\overline{T})$ the constant defined in the existence part of the theorem for the ball in which we perform the contraction. We have from the semigroup $P_t$ in the interval $[0,\overline{T}]$ that
$$
\|P_t (f_{t_i})\|_{X_{\overline{T}}}\lesssim re^{-\overline{T}}.
$$
We choose then $\overline{T}>0$ such that it holds $\|P_t (f_{t_i})\|_{X_{\overline{T}}}\leq \gamma(T^0)$, for $i\in \N$.\\
We can now repeat the same existence scheme and we obtain that $f$ exists at least on $[t_i,t_i+\overline{T}],\ \forall i\in\N$ contradicting the maximal extension of the interval. Thus, for $f^0\in L^2$ we must have
$$
\sup_{T(f^0)-\epsilon<t<T(f^0)}\|f_t\|_{L^2_M}=\infty.
$$
This concludes the proof.
\end{proof}
\begin{theorem}
Given the solution $f\in X_T$ of the truncated equation from Theorem \ref{thm:approx-plm}, we have that $f(t; f^0) \in L_+^2$ for $t \in [0,T(f^0))$ provided that $f^0 \in L_+^2$.
\end{theorem}
\begin{proof}
Consider $f^0\in L^2_+,$ and let $T\in [0,T^0)$ be arbitrary. Then there is a constant $\omega:=\omega(T)>0$ such that, for $t\in [0, T]$,
\begin{align*}
    \left|\sum_{n=1}^M\int |v-w|\chi_R(w)\chi_R(v)f_n(w)dw\right|\leq \sum_{n=1}^M C_R\|f_n\|_{X_T}:=\omega(T),\ a.e.\ v\in\R^d.
\end{align*}
Now, for $g\in L^2$ and $0\leq t\leq T$ set for $m=1,..,M$
$$
\overline{Q}^m_t(g):=dg_m+Q_m^R(g,g)+\omega(T)_Rg_m,
$$
so, for $t\in[0,T],\ a.e.v\in\R^d$ and $g\in L^p_+$, we have $\overline{Q}(g,g)\geq0$.\\
We consider now the evolution equation
\begin{align}\label{poseq}
    \Dot{g}+\left(\omega(T)+\mathcal{L}\right)g=\overline{Q}(g,g),\ t\in[0,T],\ g(0)=f^0.
\end{align}
We note that $f$ is a solution of such equation and that equation $(\ref{poseq})$ can be solved by the method of successive approximations exactly as in the existence theorem. So we consider
$$
\Theta(g):=P_\omega f^0+\int_0^\cdot P_{\cdot-s}^\omega (\overline{Q}(g,g) ds,\ \text{for}\ g\in X_T
$$where $P_\omega(t):=e^{-\omega t}P_t,\ t\geq0$.\\
In the same fashion as Theorem \ref{thm:approx-plm}, $\Theta$ is a contraction in a suitable ball $B_T\subset X_T$, centered in $P_\cdot f^0$, into itself. Taking $T$ small enough we can assume that also $f\in B_T$. 
Therefore, since there must be existence and uniqueness of a solution we can find a sequence $g_i$, determined by
$$
g_0=f^0,\ g_{i+1}=\Theta(g_i),\ i\in\N,
$$
such that converges to $f$ in $X_T$. This implies that $g_i\rightarrow f$ in $L^2_M$ for $t\in(0,T]$.\\
Since the semigroup $P_t$ preserves positivity, by induction we get that $g_i\in L^2_+$ for $t\in[0,T]$ and $i\in\N$. This means that also $f_t\in L^2_+$ for all $t\leq T$ since the space is closed.\\
Consider now a $\overline{T}\leq T^0$, the maximal time for which $f$ is positive on $[0,\overline{T}]$, then we claim that $\overline{T}=T^0$. Otherwise, we consider now the same equation $\eqref{poseq}$ re-scaled in time
\begin{align*}
    \Dot{g}+\left(\omega+\mathcal{L}\right)g=\overline{Q}(g,g)(t+\overline{T}),\ t\in[0,T^0-\overline{T}],\ g(0)=f_{\overline{T}},
\end{align*}
and this would lead to a contradiction on the maximality of T. Since $T^0$ was arbitrary in the maximal interval of definition we deduce that $f\in L^2_+,\ \forall t\in [0,T^0)$.
\end{proof}
We wish now to understand better the regularity of the truncated solution and the estimate on such computation. To this end we want to use the a priori estimate, but in doing so we first need to establish regularity deriving directly from the semigroup and the nonlinearity.\\
Following from Proposition \ref{ppn:reg-semigp} we can derive some regularity on $f$ and its derivatives in $L^2$ space. In particular, with $f_0\in L^2_M$, we have
$$
f\in L^\infty([0,T],H^2(\R^d))\cap W^{1,2}_{\text{loc}}([0,T]\times\R^d).
$$
This is still not enough and we want to control also the moments of the solution and to this end it is not enough to have $f_0\in L^2_m$, so we first consider an initial condition such that $f_0\in L^{1}_{p,M}:=L^1_M(\brak{v}^pdv)$ for every $p\geq0$ and also in $L^2_{+,M}$, we show that this is enough to get that the solution 
$$
f^R\in L^\infty([0,T],L^1_{p,M}\cap L^2_{+,M}(\R^d)),
$$
And this holds for every $p\geq0$.\\
\\
We want to establish a Gronwall relationship, so we integrate the function $f_m$ (we drop the $R$ for readability) for all $t>0$. We note first that following the step of Theorem \ref{thm:approx-plm} it is easy to construct the same function adding the condition that $\int |f|dv$ must be bounded.\\ With this in mind, using the consistency\cite{ConsistSem} of the operator $P_t$, we can write:
\begin{align*}
    \|f_m(t)\|_{1}&\underbrace{=}_{\text{positivity}}\int f_m(t)dv=\int P_tf^0_mdv+\int \int_0^tP_{t-s}\left(df_s+Q_m^R(f,f)_s\right) dsdv\\
    &\underbrace{\leq}_{\text{property\ of}\ P_t} \int f_m^0dv+d\int_0^t \|f_m(s)\|_1ds +T\sup_t \left(\int Q^R_m(f,f)dv\right)
\end{align*}
analogous to the proof of Lemma \ref{lem:conserved} we have $\sup_t \sum_{m=1}^M \int Q^R_m(f,f)dv\leq 0$ and thus summing on $m$ and taking the sup in $t\leq T$ we get the bound of the $L^1_M$ norm, independent on $R$.\\
\\
Establish the $L^1$ bound, we use this results to prove that, depending on $R$ this time, we have bound on  the $p$-th moment for every $p$. To this end consider:
\begin{align*}
    &\int f_m(t)\brak{v}^p dv \leq \iint K(t,w)f_m^0(e^tv-w)\brak{v}^p dwdv+\int_0^t\iint K(s,w)Q^R_m(f,f)(e^sv-w)\brak{v}^pdvdwds\\
    &\underbrace{\leq}_{e^tv-w=z}\int K(t,w)\brak{w}^pdw \int C_t f_m^0(z)\brak{z}^p dz+\int_0^t\int K(s,w)\brak{w}^pdw\int C_s Q^R_m(f,f)(z)\brak{z}^pdzds
\end{align*}
Recall that 
\begin{align*}
K(t,w):=\frac{1}{(2\pi(e^{2t}-1))^{\frac{d}{2}}}e^{-\frac{|w|^2}{2\pi(e^{2t}-1)}},
\end{align*}
and that the truncated nonlinearity has a term of the form $\chi_R(z)\chi_R(w)|z-w|$ and as such we get
\begin{align*}
    \int f_m(t)\brak{v}^p dv &\underbrace{\leq}_{\int K(t,w)\brak{w}^pdw<\infty\ \forall p}C_T\|f^0\|_{p,M}^1+C(R)\int_0^tC_s\sum_{n=1}^{m-1}\iint f_n(\varphi(v,w))f_m(w)dwdvds\\
    &\leq C_T\|f^0\|_{p,M}^1+C(R)C_T\sum_{n=1}^{m-1}\sup_t\|f_n\|_1\|f_m\|_1\\
    &\leq C_T\|f^0\|_{p,M}^1+C_T(R,\sup_t\|f\|_{M}^1,M)\sup_t\|f\|^1_{p,M}.
\end{align*}
Summing on $m=1,...,M$, taking the sup over $t\leq T$, we obtain the bound depending on the initial condition, the $L^1$ norm of the solution and the truncation $R$.\\
\\
We show now that a mild solution is indeed a weak solution for our equation and use this to enhance the regularity of the time derivative of the solution.
\begin{lemma}[$L^p$ mild solutions are weak solutions]
Consider $f\in X_T$ a mild solution for equation $(\ref{main-eq-trunc1})$ corresponding to the initial value $f_0$, then $f$ is a weak solution in the sense of Definition \ref{def:weak-sol}.
\end{lemma}
\begin{proof}
Fix $T>0$ an a time $t\in[0,T]$, fix $\phi$ a test function. Consider now the weak formulation for equation $\eqref{main-eq-trunc1}$:
$$
\brak{f_t,\phi}= \brak{f_0,\phi}+\int_0^t\brak{f_s,\mathcal{L}^*\phi}ds+\int_0^t\brak{Q_s(f,f),\phi}ds,
$$
where $\mathcal{L}^*\phi=-v\cdot\nabla\phi+\Delta\phi$. Notice that
$$
\brak{f_t,\mathcal{L}^*\phi}=\brak{\mathcal{L}(f_t)+df_t,\phi}.
$$
Under this definition and observation we consider now $f$ mild solution as before
$$
f_t=P_tf_0+\int_0^tP_{t-s}f_sds+\int_0^tP_{t-s}Q_s(f,f)ds
$$
ans analyze the well defined quanity
\begin{align*}
\int_0^t\brak{f_s,\mathcal{L}^*\phi}ds&=\int_0^t\brak{P_sf_0+\int_0^sP_{s-r}f_rdr+\int_0^sP_{s-r}Q_r(f,f)dr,\mathcal{L}^*\phi}ds\\
&=\underbrace{\int_0^t\brak{P_sf_0,\mathcal{L}^*\phi}ds}_{A}+\underbrace{\int_0^t\brak{\int_0^sP_{s-r}f_rdr,\mathcal{L}^*\phi}ds}_{B}+\underbrace{\int_0^t\brak{\int_0^sP_{s-r}Q_r(f,f)dr,\mathcal{L}^*\phi}ds}_{C}.
\end{align*}
Now we analyze, using the regularity result on $f$ derived in the previous sections and using \cite{MPV}, each of the component on the right hand side.\\
\textbf{A:} 
\begin{align*}
    \int_0^t\brak{P_sf_0,\mathcal{L}^*\phi}ds&=\int_0^t\brak{\mathcal{L}P_sf_0,\phi}ds+\int_0^t\brak{P_sf_0,\phi}ds\\
    &=\brak{\int_0^t\frac{d}{ds}P_sf_0ds,\phi}+\int_0^t\brak{P_sf_0,\phi}ds\\
    &=\brak{P_tf_0,\phi}-\brak{f_0,\phi}+\int_0^t\brak{P_sf_0,\phi}ds
\end{align*}
\textbf{B:} 
\begin{align*}
\int_0^t\brak{\int_0^sP_{s-r}f_rdr,\mathcal{L}^*\phi}ds&=
\int_0^t\brak{\int_0^s\mathcal{L}P_{s-r}f_rdr,\phi}ds+
\int_0^t\brak{\int_0^sP_{s-r}f_rdr,\phi}ds.\\
&=\int_0^t\brak{\int_r^t\mathcal{L}P_{s-r}f_rds,\phi}dr+
\int_0^t\brak{\int_0^sP_{s-r}f_rdr,\phi}ds\\
&=\int_0^t\brak{\int_r^t\frac{d}{ds}P_{s-r}f_rds,\phi}dr+
\int_0^t\brak{\int_0^sP_{s-r}f_rdr,\phi}ds\\
&=\int_0^t\brak{P_{t-r}f_r,\phi}dr-\int_0^t\brak{f_r,\phi}dr+
\int_0^t\brak{\int_0^sP_{s-r}f_rdr,\phi}ds\\
\end{align*}
\textbf{C:} 
\begin{align*}
\int_0^t\brak{\int_0^sP_{s-r}Q_r(f,f)dr,\mathcal{L}^*\phi}ds&=\int_0^t\brak{\int_0^s\mathcal{L}P_{s-r}Q_r(f,f)dr,\phi}ds+\int_0^t\brak{\int_0^sP_{s-r}Q_r(f,f)dr,\phi}ds\\
&=\int_0^t\brak{\int_0^s\mathcal{L}P_{s-r}Q_r(f,f)dr,\phi}ds+\int_0^t\brak{\int_0^sP_{s-r}Q_r(f,f)dr,\phi}ds\\
&=\int_0^t\brak{\int_r^t\mathcal{L}P_{s-r}Q_r(f,f)ds,\phi}dr+\int_0^t\brak{\int_0^sP_{s-r}Q_r(f,f)dr,\phi}ds\\
&=\int_0^t\brak{P_{t-r}Q_r(f,f)-Q_r(f,f),\phi}dr+\int_0^t\brak{\int_0^sP_{s-r}Q_r(f,f)dr,\phi}ds.
\end{align*}
To prove all the following equality we have used the fact that, if $f\in D(\mathcal{L})$, then $P_tf\in D(\mathcal{L})$ and in that case 
$$
\frac{d}{dt}P_tf=\mathcal{L}P_tf=P_t\mathcal{L}f.
$$
Thanks to the reguliraty on $f$ we can apply this result, thus putting all together we get:
\begin{align*}
 \int_0^t\brak{f_s,\mathcal{L}^*\phi}ds&=\brak{P_tf_0,\phi}-\brak{f_0,\phi}+\int_0^t\brak{P_sf_0,\phi}ds\\
 &+\int_0^t\brak{P_{t-r}f_r,\phi}dr-\int_0^t\brak{f_r,\phi}dr+
\int_0^t\brak{\int_0^sP_{s-r}f_rdr,\phi}ds\\
&+\int_0^t\brak{P_{t-r}Q_r(f,f),\phi}dr-\int_0^t\brak{Q_r(f,f),\phi}dr+\int_0^t\brak{\int_0^sP_{s-r}Q_r(f,f)dr,\phi}ds.
\end{align*}
For which we derive, using the equality of the mild solution:
\begin{align*}
 \brak{f_0,\phi}&+\int_0^t\brak{f_s,\mathcal{L}^*\phi}ds+\int_0^t\brak{Q_r(f,f),\phi}dr=\underbrace{\brak{P_tf_0,\phi}+\int_0^t\brak{P_{t-r}f_r,\phi}dr+\int_0^t\brak{P_{t-r}Q_r(f,f),\phi}dr}_{:=\brak{P_tf_o+\int_0^tP_{t-r}f_rdr+\int_0^tP_{t-r}Q_r(f,f)dr,\phi}=\brak{f_t,\phi}}\\
 &+\underbrace{\int_0^t\brak{P_sf_0,\phi}ds +\int_0^t\brak{\int_0^sP_{s-r}f_rdr,\phi}ds+\int_0^t\brak{\int_0^sP_{s-r}Q_r(f,f)dr,\phi}ds}_{:=\int_0^t\brak{f_s,\phi}ds}-\int_0^t\brak{f_s,\phi}ds.
\end{align*}
Thus $\brak{f_0,\phi}+\int_0^t\brak{f_s,\mathcal{L}^*\phi}ds+\int_0^t\brak{Q_r(f,f),\phi}dr=\brak{f_t,\phi}$, concluding the proof.
\end{proof}
To conclude we need to extend the regularity on the time derivative and prove that
$$
\partial_t f_m\in L^2([0,T]\times\R^d).
$$
Analyzing the equation solved by $f^R$, proving that $\partial_t f^R_m$ is in $L^2([0,T]\times \R^d)$ boils down to prove that the quantity
$$
v\cdot \nabla f_m\in L^2([0,T]\times\R^d)\text{  i.e. }\ \int_0^T \int |v\cdot \nabla f_m(s)|^2ds<\infty.
$$
In order to to this we use the mild formulation of $f$, the explicit formulation of the nonlinearity $Q^R_m$ and the bound we have previously derived. We reduce again the problem, using Cauchy-Schwartz and the inequality on the weight to prove that
$$
\int_0^t \int|\nabla f_m|^2 \brak{v}^2dvds <\infty.
$$
We consider:
\begin{align*}
    \int_0^t \int|\nabla f_m|^2 \brak{v}^2dvds&\leq \underbrace{\int_0^t\int|\nabla P_tf_0|^2 \brak{v}^2dsdv}_{A}+\underbrace{\int_0^t \int|\nabla \int_0^s P_{s-r}Q_m^R(f,f)dr|^2 \brak{v}^2dvds}_{B}
\end{align*}
we work separately on the two term:
\begin{align*}
    A&:=\int_0^t\int|\nabla \int K(t,w)f_0(e^tv-w)dw|^2\brak{v}^2dvds\\
    &\leq e^Tc_d\int_0^t\int\int K(t,w)|\nabla f_0(e^tv-w)|^2\brak{v}^2dwdvds\\
    &\leq e^Tc_d\int_0^t\int K(t,w)\int|\nabla f_0(e^tv-w)|^2\brak{v}^2dvdwds\\
      &\lesssim  \int_0^t\int K(t,w)\int|\nabla f_0(e^tv-w)|^2(\brak{e^tv}^2+\brak{w}^2)dvdwds\\
    &\lesssim \|\nabla f_0\|_{2,2}+\|\nabla f_0\|_2.
\end{align*}
while for the other terms, we use intensively the form of the nonlinearity $Q^R$. In particular we use the bound obtained with $\chi_R(v)$ and $\chi_R(w)$ to control the weight and conclude.
\begin{align*}
    &B\lesssim\int_0^t \int_0^s \underbrace{\left(\int |\nabla P_{s-r}Q_m^R(f,f)|^2 \brak{v}^2dv\right)}_{C}drds
\end{align*}
and we compute directly on the term $C$.
\begin{align*}
    \left(\int |\nabla P_{s-r}Q_m^R(f,f)|^2 \brak{v}^2dv\right)&\leq e^T\int |P_{s-r}\nabla Q_m^R(f,f)|^2 \brak{v}^2dv\\
    &=e^T\int |P_{s-r}(\left(Q_m^R(\nabla f,f)+Q_m^R(f,\nabla f)\right)|^2 \brak{v}^2dv\\
    &\lesssim \int P_{s-r}|(\left(Q_m^R(\nabla f,f)+Q_m^R(f,\nabla f)\right)|^2 \brak{v}^2dv
\end{align*}
Since the two addenda of the sum are quite similar we focus only on one of them. In particular we focus on each single term one for the "positive" part, the negative part being easier.
\begin{align*}
    &\int \int K(s-r,w)\left|\int \nabla f_n(h(v,w,z) f_{m-n}(z)|z|\chi_R(e^tv-w)\chi_R(z)dz\right|^2dw\brak{v}^2dv\\
    &\leq \int \int K(s-r,w)\left|\int \nabla f_n(h(v,w,z) f_{m-n}(z)|z|\chi_R(e^tv-w)\chi_R(z)dz\right|^2dw\brak{v}^2dv\\
    &\leq C_R^2 \|f\|_1\int \int K(s-r,w)\int |\nabla f_n(h(v,w,z)|^2 f_{m-n}(z)dzdw\brak{v}^2dv\\
    &\underbrace{\lesssim}_{\text{change\ variables}}C_R^2 \|f\|_1\int \int K(s-r,w)\int |\nabla f_n(\gamma)|^2 f_{m-n}(z)\brak{\gamma}^2\brak{w}^2\brak{z}^2d\gamma dzdw\\
    &\lesssim e^T \|\nabla f_n(s)\|_{2,2}^2\|\|f_n\|_{1,2}^2\|_{\infty,T}.
\end{align*}
Putting this in $(B)$ we get
$$
\int_0^t \int_0^s\left(\int |\nabla P_{s-r}Q_m^R(f,f)|^2 \brak{v}^2dv\right)drds\leq C(R,M,T,\|f\|_1,\|f\|_{1,2})\int_0^t \int_0^s\|\nabla f(s)\|_{2,2}^2ds\ dt.
$$
Summing on $M$ we get:
$$
\sum_{m}\int_0^t \int|\nabla f_m|^2 \brak{v}^2dvds \leq C(f_0)+C(R,M,T,\|f\|_1,\|f\|_{1,2})\int_0^t \int_0^s\|\nabla f(s)\|_{2,2}^2ds\ dt
$$
Assuming that $f_0$ is in $H^1_1\cap L^1_2$ we have proven:
$$
f_m\in H^1([0,T]\times\R^d), \forall m=1,...,M.
$$
This extended regularity for the truncated solution $f^R$ make it possible to prove rigorously the a priori estimate for the solutions of the approximated problems, obtaining bound that are independent of $R$.
\begin{corollary}
In the same hypothesis of the last few theorems, the a priori estimate are true for the solution of the approximating problems, with constants independent from $R$, as soon as $f^0$ is regular enough.
\end{corollary}
\begin{proof}
Under the hypothesis, and thanks to the previous lemmas we satisfy the requirement of the a priori estimate and thus conclude the proof. In fact using the established $L^1_k$ bounds and the fact that $f\in H^1([0,T]\times\R^d)$ we can make rigorous the arguments of the a priori estimates lemmas and then proceed to obtain the regularity 
$$
f^R\in L^\infty([0,T],H^p_k(\R^d)),\ \forall n\in\N, \forall k\geq0,
$$
with bound independent of $R$.
\end{proof}
We conclude this section with the following:
\begin{theorem}[global existence]
Assume that the initial condition is positive and in $L^2_M$ such that Lemma \ref{lem:H_1-trunc} and Theorem \ref{thm:approx-plm} holds, then the solution $\tilde{f}$ of the approximated problem is global in time, i.e. $T(f^0)=\infty$.
\end{theorem}
\begin{proof}
Thanks to Theorem \ref{thm:approx-plm} we know that the maximal interval of the solution $\tilde{f}$ is open in $R^+$ and of the form $[0,T^0)$.\\
Under the hypothesis of regularity of the initial condition, and thanks to Lemma \ref{lem:H_1-trunc}, we have that
$$
\|\tilde{f}_t\|_{L^2_M}\leq C(T),\ a.e.\ t\in[0,T^0)\cap[0,T].
$$
Thus, since being $T^0$ finite means that $\sup_{t\in[T^0-\epsilon,T^0)}\|\tilde{f}_t\|_{L^2_M}=\infty$, this cannot be happening and this implies that $T^0=\infty$ and the solution is global.\\
This implies that, for each $T>0$, independently of $R$, under the regularity of initial condition, we have a weak solution of equation $\eqref{main-eq-trunc1}$ on $X_T$.
\end{proof}

\section{Existence and uniqueness for the full System}
Right now we have proven that, fixing $T$ finite, there exists a sequence of function $(f^R)_R$, solutions of $\eqref{main-eq-trunc1}$.\\
Under suitable initial condition $(f^0_m)_{m=1}^M$, we prove now the weak convergence of $f^R$ to a map $f$ in $L^2\cap L^1$ such that solves in the weak sense $\eqref{main-eq1}$.\\
To prove existence for the full problem we begin by recalling the definition of weak solution we are going to use:
\begin{definition}
Fix $T>0$, and $f^0_m\in L^1_2\cap L^2$. A solution for equation $(\ref{main-eq1})$ is a set of functions $f_m\in L^\infty(0,T;L_2^1(\R^d))$ indexed by the masses $m=1,...,M$ such that $\forall m,\ Q_m(f,f)\in L^\infty(0,T;L^1(\R^d))$ and the following holds:
\begin{align*}
    \sum_{m=1}^M\brak{f_m(t),\phi}-\sum_{m=1}^M\brak{f_m^0,\phi}&=\int_0^t \sum_{m=1}^M\brak{f_m(s),\Delta\phi}ds-\int_0^t \sum_{m=1}^M\brak{f_m(s),v\cdot\nabla\phi}ds\\
    &+ \int_0^t \sum_{m=1}^M\brak{Q_m(f(s),f(s)),\phi}ds,
\end{align*}
for a.e. $t\in [0,T]$, and $\forall\phi\in C^\infty_c(\R^d)$. \end{definition}
We note that if a weak solution is sufficiently smooth and satisfy suitable decay for large $|v|$, then it is also a classical solution.
With everything recalled we state the following.
\begin{theorem}
For every $T>0,M>0,\ f_0\geq0$ such that $f_0\in L^1_{2,M}\cap L^2_M$ there exists at least one nonnegative weak solution $f\in L^\infty([0,T],L^1_2(\R^d))$ with initial condition $f(0)=f_0$.\\
If instead $f_0\in L^1_{2,M}\cap H^1_{1,M}$, then for every $t_0>0$ we have $f\in C^1_b([t_0,T],C^{\infty}(\R^d))$ with rapid decay for large $|v|$.
\end{theorem}
We prove here the existence, postponing the uniqueness to the next section.
\begin{proof}(Existence) We start with an initial condition $f_0\in C^\infty_c(\R^d)$. Using the results obtained in the previous section we have establish the $L^1_k$ bounds and the fact that the sequence $f_R\in H^1([0,T]\times\R^d)$. With this we can make rigorous the a priori estimates and get $f_R\in L^\infty([0,T],H^n_k(\R^d)),\ \forall n\in\N,k\geq0$, with bounds independent from $R$.\\ 
$f^R\in L^\infty([0,T],L^2)$, and thanks to the a priori estimate we have that, independently of $R$, $f^R_t\in H^1(\R^d)\cap L^2(\R^d,|v|^2dv)$ and the sequence is equibounded. Analogously we have that, thanks to the regularity of the semi group $P_t$, the solution $f^R$ lies in the space $H^1(0,T;H^{-1})$.\\
Using now the Weighted Aubin-Lions Lemma \ref{lem:aubin} we have that, up to subsequence:
\begin{enumerate}
    \item $f^R\rightharpoonup f$ in $H^1$;
    \item $f^R\rightarrow f$ in $L^2$;
    \item up to subsequence $f^R\rightarrow f$ pointwise a.e.,
\end{enumerate}
and the a priori estimate pass to the limit function $f$.\\
This is enough to allow us to pass to the limit as $R\rightarrow\infty$ in the weak formulation of the system and show that the limit solutions satisfy the equation with the full kernel.\\
Finally we consider $f_0\in L^1_2\cap L^2$ and we take a sequence $f^n_0\in C^\infty_0$ converging to $f_0$ in $L^1_2\cap L^2$, obtaining the bound in $L^1_2$ by Fatou's lemma. Then, since the constants in the moments and energy are independent on $n$ we can pass to the weak $L^1-$limit in the equations obtaining a solution $f$ as expected.\\
To study the regularity of the solution with data in $H^1_1$ we consider that this is enough to get the parabolic regularity $f\in H^1([0,T]\times\R^d)$ for any $T>0$ and using the fact that we have the bound $f\in L^\infty([0,T],L^1_k(\R^d))$ we can make rigorous the argument of the a priori estimates and proceed to obtain the desired regularity.
\end{proof}
\subsection{Weak Uniqueness}
Consider now two weak solutions $f,g$ of equation $(\ref{main-eq1})$ in the sense of Definition \ref{def:weak-sol}, with initial conditions $f^0_m=g^0_m\in H^1_1\cap L^1_2$ and nonnegative for every $m$. We know they are actually classical solutions. We want to show that, as soon as they start with the same initial, regular condition, then the two solutions must be equal a.e. in $v\in\R^d$ for every fixed time $T\geq0$. \\
Consider for this the set of equation solved by $h_m(t,v):=f_m(t,v)-g_m(t,v),\ \forall m=1,...,M$ and call $H_m(t,v):=f_m(t,v)+g_m(t,v)$.
\begin{align}\label{eq-for-h_m}
    &\begin{cases}
    \partial_t h_m(t,v)=\Delta h_m +\text{div}(vh_m)+\frac{1}{2}\left(Q_m(h,H)+Q_m(H,h)\right)\\[5pt]
    h_m|_{t=0}=0
    \end{cases}
    \end{align}
    where we denote
    \begin{align*}
    Q_m(h,H)&:=\sum_{n=1}^{m-1}\int\left(\frac{m}{n}\right)^2 h_n(\varphi(v,w))H_{m-n}(w)|v-w|dw -2\sum_{m=1}^M h_m(v)\int H_n(w)|v-w|dw,\\
    Q_m(H,h)&:=\sum_{n=1}^{m-1}\int\left(\frac{m}{n}\right)^2  H_n(\varphi(v,w))h_{m-n}(w)|v-w|dw -2\sum_{m=1}^M H_m(v)\int h_n(w)|v-w|dw.
\end{align*}

We consider now the function $\psi_\varepsilon(x)$ an approximation of the function $\text{sgn}(x)$ such that 
\begin{align*}
    &\psi_\varepsilon(x):=\begin{cases}
       -1,\ &x\leq-\varepsilon\\
       x/\ep,\ &-\varepsilon<x< \varepsilon\\
       1,\ &x\ge \varepsilon
    \end{cases}
\end{align*}
We consider the weight $\brak{v}^2:=1+|v|^2$ and thus the function 
$\Psi_\varepsilon(h_m,v):=\psi_\varepsilon(h_m)\brak{v}^2$, multiplying it to the equations \eqref{main-eq1} solved by $h_m, \forall m$ and integrating by part. Defining also 
\begin{align*}
\chi_{\varepsilon}(x):=\begin{cases}
    &1,\ \text{if}\ |x|< \varepsilon\\
    &0,\ \text{otherwise},
\end{cases}
    \end{align*}
we obtain:
\begin{align}\label{hormander}
    \partial_t\int h_m\psi_\varepsilon(h_m)\brak{v}^2dv&=\int (\partial_th_m)\psi_\varepsilon(h_m)\brak{v}^2dv+\frac{1}{\varepsilon}\int (\partial_th_m)\chi_\varepsilon(h_m)h_m\brak{v}^2dv\\
    &=: A_1+B_1.\nonumber
\end{align}
Now we can work on the quantity $A_1$, using the equation \eqref{main-eq1},
\begin{align*}
    A_1:&=\int\Delta h_m\psi_\varepsilon(h_m)\brak{v}^2+  \div(vh_m)\psi_\varepsilon(h_m)\brak{v}^2+\frac{1}{2}\left(Q_m(h,H)+Q_m(H,h)\right)\psi_\varepsilon(h_m)\brak{v}^2dv\\
    &=-\frac{1}{\varepsilon}\int_{\{|h|<\varepsilon\}}|\nabla h_m|^2\brak{v}^2dv + d\int h_m\psi_\varepsilon(h_m)\brak{v}^2dv+\underbrace{\int \left(\nabla h_m \cdot v\right)\psi_\varepsilon(h_m)\left(|v|^2-1\right)dv}_{A_2}\\
    &+\underbrace{\frac{1}{2}\int\left(Q_m(h,H)+Q_m(H,h)\right)\psi_\varepsilon(h_m)\brak{v}^2dv}_{A_3},
\end{align*}
and we can then drop the first negative term. For $A_2$ we define first the following function:
\begin{align*}
    &\zeta_\varepsilon(x):=\begin{cases}
       -x+\ep/2,\ &x\leq-\varepsilon\\
       x^2/(2\ep),\ &-\varepsilon<x< \varepsilon\\
       x-\ep/2,\ &x\ge\varepsilon
    \end{cases}
\end{align*}
with $\nabla \zeta_\ep=\ovl\psi_\ep$ a.e. and $\zeta_\ep\to|x|$ as $\ep\to0$, and we can write
\begin{align*}
    A_2=\int \left(\nabla \zeta_\ep(h_m) \cdot v\left(|v|^2-1\right)\right)dv&=-\int \zeta_\ep(h_m) \div(v\left(|v|^2-1\right))dv\\
    &=-\int \zeta_\varepsilon(h_m)\left(d\left(|v|^2-1\right)+2|v|^2\right)dv.
\end{align*}
Now we get for $A_3$:
\begin{align*}
    A_3:&=\frac{1}{2}\sum_{n=1}^{m-1}\iint\left[h_n(\varphi(v,w))H_{m-n}(w)+h_{m-n}(w)H_{n}(\varphi(v,w))\right]|v-w|\psi_\varepsilon(h_m(v))\brak{v}^2dwdv\\
    &-\sum_{n=1}^{M}\iint  \left[h_m(v)H_{n}(w)+h_{n}(w)H_{m}(v)\right]|v-w|\psi_\varepsilon(h_m(v))\brak{v}^2dwdv.
\end{align*}
Next, we turn to the term $B_1$ of \eqref{hormander}. Since $h_m$ appears in the integrand, the set of $v$ such that $h_m(v)=0$ does not contribute to the integral, hence 
\begin{align*}
    \left|\frac{1}{\varepsilon}\int (\partial_th_m)\chi_\varepsilon(h_m)h_m\brak{v}^2dv\right| &=\left| \frac{1}{\ep}\int_{\{v:\, 0<|h_m(v)|<\ep\}}(\partial_th_m)h_m\brak{v}^2dv\right| \\
    &\le\int_{\{v:\, 0<|h_m(v)|<\ep\}}|\partial_th_m|\brak{v}^2dv.
\end{align*}
Since $h_m$ is continuous, $|h_m|$ is also continuous sending open sets to open sets, hence $\{0<|h_m(v)|<\ep\}\downarrow\emptyset$ as $\ep\to0$. We now claim that $\partial_th_m(t)\brak{v}^2\in L^1$ for every $t\in[0,T]$. It will imply that the last integral goes to $0$ as $\ep\to0$. By linearity $\partial_th_m=\partial_tf_m-\partial_tg_m$, hence it is enough to verify this claim separately for $f_m$ and $g_m$. Let us take $f_m$ and prove $\partial_t f_m(t)\brak{v}^2\in L^1$ for a.e. $t$. Note that
\begin{align*}
    \int |\partial_t f_m| \brak{v}^2dv &=\int \left|\Delta f_m +\text{div}(vf_m)+Q(f_m,f_m)\right| \brak{v}^2 dv\\
    &=\int \left|\Delta f_m +df_m+v\cdot\nabla f_m+Q(f_m,f_m)\right| \brak{v}^2 dv.
\end{align*}
By Sobolev embedding (which extends to the weighted case), if $f_m\in H^n_k$ for $n,k$ suitably large (in particular, since $f_m$ is a Schwartz function), then the above integral is finite.

Now we sum \eqref{hormander} over $m=1,..,M$ and we get
\begin{align*}
    \partial_t\sum_{m=1}^M\int h_m\psi_\varepsilon(h_m)\brak{v}^2dv&=\sum_{m=1}^M\int (\partial_th_m)\psi_\varepsilon(h_m)\brak{v}^2dv+\sum_{m=1}^M\frac{1}{\varepsilon}\int (\partial_th_m)\chi_\varepsilon(h_m)h_m\brak{v}^2dv\\
    &=:A_1^M+B_1^M.
\end{align*}
We already know that $B_1^M$ is negligible. Further, we have that 
\begin{align*}
    A^M_1&=\sum_{m=1}^M\int\Delta h_m\psi_\varepsilon(h_m)\brak{v}^2+ \div(vh_m)\psi_\varepsilon(h_m)\brak{v}^2dv+\frac{1}{2}\int\left(Q_m(h,H)+Q_m(H,h)\right)\psi_\varepsilon(h_m)\brak{v}^2dv\\
    &\leq d\int \sum_{m=1}^M\left(h_m\psi_\varepsilon(h_m)\right)\brak{v}^2dv+\underbrace{\sum_{m=1}^M \int \left(\nabla h_m \cdot v\right)\psi_\varepsilon(h_m)\left(d\left(|v|^2-1\right)+2|v|^2\right)dv}_{A^M_2}\\
    &+\underbrace{\frac{1}{2}\sum_{m=1}^M\int\left(Q_m(h,H)+Q_m(H,h)\right)\psi_\varepsilon(h_m)\brak{v}^2dv}_{A^M_3}
\end{align*}
before estimating the nonlinear part we note that, using the estimate for each $m$, as $\varepsilon\rightarrow0$, under our regularity assumption, we get
\begin{align*}
    &\partial_t\sum_{m=1}^M\int h_m\psi_\varepsilon(h_m)\brak{v}^2dv\rightarrow\partial_t\sum_{m=1}^M\int |h_m|\brak{v}^2dv;\\
    &A^M_2\rightarrow -\sum_{m=1}^M \int |h_m|
    \left(d\left(|v|^2-1\right)+2|v|^2\right)dv\leq \sum_{m=1}^M C_d\int |h_m|\brak{v}^2.
\end{align*}
So we miss to estimate the nonlinear term $A^M_3$ and pass to the limit:
\begin{align*}
    A_3^M:&=\frac{1}{2}\sum_{m=1}^M\sum_{n=1}^{m-1}\iint \left(\frac{m}{n}\right)^2\left( h_n(\varphi(v,w))H_{m-n}(w)+h_{m-n}(w)H_{n}(\varphi(v,w))\right)|v-w|\psi_\varepsilon(h_m(v))\brak{v}^2dwdv\\
    &-\sum_{m=1}^M\sum_{n=1}^{M}\iint \left(h_m(v)H_{n}(w)+h_{n}(w)H_{m}(v)\right)|v-w|\psi_\varepsilon(h_m(v))\brak{v}^2dwdv
\end{align*}

By the change of variable  
$$
\begin{cases}
    &v'= \varphi(v,w)\\        
    & w=w
\end{cases},\ |J(v,w)|=\frac{n}{m}.
$$
and naming the variable $v'$ again by $v$, we have that
\begin{align*}
    A_3^M&= \frac{1}{2}\sum_{m=1}^M\sum_{n=1}^{m-1}\iint  \left(h_n(v)H_{m-n}(w)+h_{m-n}(w)H_{n}(v)\right)|v-w|\psi_\varepsilon(h_m(v))\brak{\frac{nv+(m-n)w}{m}}^2dwdv\\
    &-\sum_{m=1}^M\sum_{n=1}^{M}\iint \left(h_m(v)H_{n}(w)+h_{n}(w)H_{m}(v)\right)|v-w|\psi_\varepsilon(h_m(v))\brak{v}^2dwdv\\
    &\underbrace{\leq}_{\text{Jensen's ineq.}\ \&\ |\psi_\ep(h_m)|\leq1}  \frac{1}{2}\sum_{n=1}^M\sum_{\ell=1}^{m-1}\iint \left(|h_n(v)|H_{\ell}(w)+|h_{\ell}|(w)H_{n}(v)\right)|v-w|(\brak{v}^2+\brak{w}^2)dwdv\\
    &-\sum_{m=1}^M\sum_{n=1}^{M}\iint |v-w| \left[h_m(v)H_{n}(w)\psi_\varepsilon(h_m(v))\brak{v}^2+h_{n}(v)H_{m}(w)\psi_\varepsilon(h_m(w))\brak{w}^2\right]dwdv\\
    &\underbrace{\le}_{\text{symmetry}}\sum_{n=1}^{M}\sum_{\ell=1}^M\iint |h_n(v)|H_{\ell}(w)|v-w|(\brak{v}^2+\brak{w}^2)dwdv\\
    &-\sum_{m=1}^M\sum_{n=1}^{M}\iint |v-w|\left[h_m(v)H_{n}(w)\psi_\varepsilon(h_m(v))\brak{v}^2+h_{n}(v)H_{m}(w)\psi_\varepsilon(h_m(w))\brak{w}^2\right]dwdv\\
    &= \iint \sum_{\ell=1}^MH_\ell(w)\sum_{n=1}^M|h_n(v)|\left(\brak{v}^2+\brak{w}^2\right)|v-w|dwdv\\
    &-\iint \sum_{m=1}^MH_m(w)\sum_{n=1}^M\left\{ h_n(v)\psi_\varepsilon(h_n(v))\brak{v}^2+h_{n}(v)\psi_\varepsilon(h_m(w))\brak{w}^2\right\}|v-w|dwdv.
\end{align*}
For the first two integrals we have (renaming $\ell$ by $m$)
\begin{align}\label{cancel-1}
\iint \sum_{m=1}^MH_m(w)\sum_{n=1}^M|h_n(v)|(\brak{v}^2+\brak{w}^2)|v-w|dwdv
\end{align}
while, using the fact that
$$
|h\psi_\ep(h)-|h||\leq |h|\chi_\varepsilon(h),
$$
the third integral can be bounded above by
\begin{align}\label{cancel-2}
    \iint \sum_{m=1}^MH_m(w)\sum_{n=1}^M \left[ -|h_n(v)|+|h_{n}|(v)\chi_\varepsilon(h_n(v))\right]\brak{v}^2|v-w|dwdv
\end{align}
while for the fourth, by the elementary inequality $|v-w|\le\langle v\rangle\langle w\rangle$, the moment bound Lemma \ref{lem:moments},  $|\psi_\ep(h_m)|\le 1$, is bounded above by
$$
 C_M\iint\sum_{m=1}^MH_m(w)\brak{w}^3\sum_{n=1}^M|h_{n}|(v)\brak{v}dwdv\le C\int \sum_{n=1}^M |h_{n}|(v)\brak{v}^2dv.
$$
We get the upper bound on $A_3^M$ by adding everything together and noting a cancellation in \eqref{cancel-1}-\eqref{cancel-2} that eliminates any higher moments in $v$,
\begin{align*}
    &A_3^M\le C_M\iint \sum_{m=1}^M H_m(w)\sum_{n=1}^M|h_n(v)|\brak{v}\brak{w}^3dvdw\\
    &+ \int \sum_{m=1}^M H_m(w)\sum_{n=1}^M\int_{\{|h_n(v)|<\varepsilon\}}|h_n(v)||v-w|dv\brak{w}^2dw+C\int \sum_{n=1}^M|h_{n}|(v)\brak{v}^2dv.
\end{align*}
Now taking the limit as $\varepsilon\rightarrow0$ we have
$$
\limsup_{\ep\to0}A^M_3\le  C\iint \sum_{m=1}^M H_m(w)\sum_{n=1}^M|h_n(v)|\brak{v}\brak{w}^3dvdw+\int \sum_{n=1}^M|h_{n}|(v)\brak{v}^2dv
$$
Consider now all the quantities together after passing to the limit, using again the moment bounds,  we have
\begin{align*}
    \frac{d}{dt}\sum_{m=1}^M\int |h_m(v)|\brak{v}^2dv
    &\leq C\int \sum_{m=1}^M|h_m(v)|\brak{v}^2dv .
\end{align*}
Now we can apply Gronwall's lemma and get that $\|f-g\|_{L^1_M(\brak{v}^2dv)}=0$, concluding the uniqueness.\qed

\begin{appendix}
\section{Particle System: Approach and Motivation}\label{appen-particle}

In this section, we provide our original motivation that gives rise naturally to the \abbr{PDE} system we study in this paper. The Smoluchowski type \abbr{SPDE} system we study is conjectured to be the scaling limit of the empirical measure of a system of diffusion particles (idealized rain droplets in the atmosphere) undergoing locally in space coagulation, while subject to an idealized form of turbulence, cf. \cite{FHP} for a discussion. More precisely, for any $d\ge1$ and $N\in\N$, consider a second-order particle system with space variable $x_i^N(t)$ in $\T^d$, velocity variable $v_i^N(t)$ in $\R^d$, mass variable $m_i^N(t)$ in a finite set $\{1,...,M\}$, and initial cardinality $N(0)=N$. Between coagulation events, the motion of an active particle obeys
\begin{equation}\label{particle-sys}
\begin{aligned}
dx^N_i(t)&=v_i^N(t)dt, \quad i\in\cN(t)\\
m_i^N(t) dv^N_i(t)&=\alpha (m_i^N(t))^{1/d}\left[\sqrt{2\mu}dB_i(t)+\sum_{k\in K}\sigma_k(x_i^N(t))\,\circ\,dW_t^k-v_i^N(t)dt\right], 
\end{aligned}
\end{equation}
(whereas after each coagulation, the velocity will be reset according to the conservation of momentum, to be precised below), where 
\begin{itemize}
\item
$\cN(t)$ denotes the set of indices of active particles in the system at time $t$, whose cardinality $|\cN(t)|\le N$,
\item
$\{B_i(t)\}_{i=1}^\infty$ is a given, countable collection of independent standard Brownian motions in $\R^d$
\item
molecular diffusivity $\mu>0$
\item
$\sigma_k(x):\T^d\to\R^d$, $k\in K$ is a given, finite (or more generally countable, subject to additional assumptions) collection of divergence free vector fields
\item
$\{W_t^k\}_{k\in K}$ is a given, finite (or countable) collection of standard Brownian motions in $\R$
\item
$\circ$ denotes Stratonovich integration.
\end{itemize}
The velocity component of the dynamics obeys  Stokes' law with random force (that depends linearly on the radius $r_i^N=(m_i^N)^{1/d}$) given by an intrinsic noise for each particle, plus a common noise of transport type
\begin{align*}
\dot{\cW}(t,x):=\sum_{k\in K}\sigma_k(x)\dot{W}_t^k
\end{align*}
that acts simultaneously on all particles, and as such is seen as an idealized form of turbulent flow in the atmosphere. We denote the $d\times d$ spatial covariance matrix of $\cW(t,x)$ by 
\begin{align*}
\cQ(x,y):=\sum_{k\in K}\sigma_k(x)\otimes\sigma_k(y).
\end{align*}
Moreover, for any fixed $x\in\T^d$ we denote the uniformly elliptic second-order divergence form operator, acting on suitable functions on $\R^d$
\begin{align*}
(\cL^{\cQ,x}_vf)(v):=\left(\mu I+\frac{1}{2}\cQ(x,x)\right)\Delta_vf(v).
\end{align*}
Note that $\cQ(x,x)$ is nonnegative definite for any $x$. For simplicity, in the sequel we consider the case that 
\begin{align}\label{enhanced-diff}
\cL^{\cQ,x}_v\equiv \kappa \Delta_v, 
\end{align}
for some constant $\kappa\ge \mu$ and all $x\in\T^d$. In view of diffusion enhancement, we have in mind that $\kappa\gg \mu$.

Each particle $i\in\cN(t)$ has a mass $m_i^N(t)\in\{1,2,..,M\}$ which changes over time according a stochastic coagulation rule to be described below. The initial mass $m_i(0)$, $i=1,...,N$, are chosen i.i.d. from $\{1,2,..,M\}$ according to a probability distribution so that $\P(m_1(0)=m)=r(m)$ with $\sum_{m=1}^Mr(m)=1$. We are also given deterministic probability density functions $g_m(x,v):\T^d\times\R^{d}\to\R_+$, $m=1,2,...,M$, satisfying additional assumptions, \footnote{those we impose on the initial condition of the \abbr{PDE} system} such that if $m_i(0)=m$ then the initial distribution of $(x_i(0),v_i(0))$ is chosen with probability density $g_m(x,v)$, independently across $i$. We denote
\begin{align*}
f^0_m(x,v)=r(m)g_m(x,v), \quad (x,v)\in\T^d\times\R^{d},\; m=1,..,M,
\end{align*}
which satisfy the same assumptions as those imposed on $\{g_m(x,v)\}_{m=1}^M$.

The rule of coagulation between pairs of particles is as follows. Let $\theta(x):\R^d\to\R_+$ be a given, $C^\infty$-smooth, symmetric probability density function in $\R^d$ with compact support in $\B(0,1)$ (the unit ball around the origin in $\R^d$) and $\theta(0)=0$. Then, for any $\ep\in(0,1)$, denote $\theta^\ep(x):\T^d\to\R_+$ by
\begin{align*}
\theta^\ep(x):=\ep^{-d}\theta(\ep^{-1}x), \quad x\in\T^d.
\end{align*}
Suppose the current configuration of the particle system is 
\begin{align*}
\eta=(x_1,v_1,m_1,x_2,v_2, m_2,...,x_N,v_N, m_N)\in(\T^d\cup\emptyset)^N\times(\R^d\cup\emptyset)^N\times\{1,...,M,\emptyset\}^N
\end{align*}
where $(x_i,v_i,m_i)$ denotes the position, velocity and mass of particle $i$, by convention if particle $i_0$ is no longer active in the system, we set $x_{i_0}=v_{i_0}=m_{i_0}=\emptyset$. Independently for each pair $(i,j)$ of particles, where $i\neq j$ each running over the index set of active particles in $\eta$, with a rate (recall \eqref{mass-const})
\begin{align}\label{rate}
s(m_i^N,m_j^N)\frac{|v_i-v_j|}{N}\theta^\ep(x_i-x_j)
\end{align}
we remove $(x_i,v_i,m_i,x_j,v_j,m_j)$ from the configuration $\eta$, and then add 
\begin{align*}
\left(x_i,\frac{m_iv_i+m_jv_j}{m_i+m_j},m_i+m_j,\emptyset,\emptyset,\emptyset\right)
\end{align*}
with probability $\frac{m_i}{m_i+m_j}$, and instead add 
\begin{align*}
\left(\emptyset,\emptyset,\emptyset, x_j,\frac{m_iv_i+m_jv_j}{m_i+m_j},m_i+m_j\right)
\end{align*}
with probability $\frac{m_j}{m_i+m_j}$. We call the new configuration obtained this way by $S^1_{ij}\eta$ and $S^2_{ij}\eta$ respectively. In words, if $(i,j)$ coagulate, we decide randomly which of $x_i$ and $x_j$ is the new position of the mass-combined particle. If the position chosen is $x_i$, then we consider $j$ as being eliminated (no longer active) and the new particle has index $i$, whereas if the position chosen is $x_j$, then we consider $i$ as being eliminated and the new particle has index $j$. On the other hand, the velocity of the mass-combined particle is obtained by the conservation of momentum as in perfectly inelastic collisions.

Note that the form of the coagulation rate \eqref{rate} is such that $(i,j)$ can coagulate only if $|x_i-x_j|\le\ep$, that is, their spatial positions have to be $\ep$-close. We are interested in the case when $\ep=\ep(N)\to0$ as $N\to\infty$, so that the interaction is not of mean-field type, but rather local, see the statement of our conjectured result below. In particular, if $\ep=O(N^{-1/d})$ then each particle typically interacts with a bounded number of others at any given time. The essential feature of our coagulation rate, that of the appearance of $|v_i-v_j|$, is inspired by the coagulation kernels used in the physics literature for describing cloud particles (which in general can depend also on $m_i,m_j$ and other physically relevant quantities), and we believe it is key to demonstrating coagulation enhancement.

For each $N\in\N$, $T\in(0,\infty)$ and $m\in\{1,..,M\}$, we denote the process of empirical measure on position and velocity of mass-$m$ particles in the system by 
\begin{align*}
\mu^{N,m}_t(dx,dv):=\frac{1}{N}\sum_{i\in\cN(t)}\delta_{x_i^N(t)}(dx)\delta_{v_i^N(t)}(dv)1_{\{m_i^N(t)=m\}}\in\cM_{1,+}(\T^d\times\R^d)
\end{align*}
where $\cM_{1,+}:=\cM_{1,+}(\T^d\times\R^d)$ denotes the space of subprobability measures on $\T^d\times\R^d$ equipped with weak topology. The choice of the initial conditions for our system implies that $\P$-a.s.
\begin{align*}
\mu_0^N(dx,dv)\to f^0(x,v)dxdv, \quad \text{as }N\to\infty
\end{align*}
where the limit is absolutely continuous. We conjecture that, under the assumption of local interaction, i.e.
\begin{align*}
\lim_{N\to\infty}\ep(N)=0, \quad \limsup_{N\to\infty}\frac{\ep(N)^{-d}}{N}<\infty.
\end{align*}
for every finite $T$, the collection of empirical measures $\{\mu^N_t(dx,dv):t\in[0,T]\}_{m=1}^M$ converges in probability, as $N\to\infty$, in $\cD([0,T],\cM_{1,+})^M$ the space of c\`adl\`ag paths taking values in $\cM_{1,+}$, equipped with the Skorohod topology, towards an absolutely continuous limit, which is the pathwise unique weak solution $\{f_m(t,x,v):t\in[0,T]\}_{m=1}^M$ of a Smoluchowski-type \abbr{SPDE} system, see \eqref{smol-spde}. The latter \abbr{SPDE} degenerates to the \abbr{PDE} system we study in this paper \eqref{main-eq-intro} when the It\^o term is switched off. While we do not provide a rigorous proof here, we sketch below a heuristic argument and postpone the rigorous derivation of the \abbr{SPDE} from the particle system to a future work. We think that this heuristic argument is sufficient to justify our interest in studying our \abbr{PDE} system. In fact, in the literature there exists specific limiting procedures cf. \cite{Galeati, FGL} that allow, in principle, to obtain the \abbr{PDE} from the \abbr{SPDE} by carefully choosing the vector fields $\sigma_k(x)$.

Let $\phi(x,v)$ be any function of class $C_c^\infty(\T^d\times\R^d)$, for any $m$ we apply It\^o's formula to the process
\begin{align*}
\langle \phi, \mu^{N,m}_t\rangle:=\frac{1}{N}\sum_{i\in\cN(t)}\phi(x_i^N(t),v_i^N(t))1_{\{m_i^N(t)=m\}}
\end{align*}
and we get 

\begin{align*}
&\langle \phi, \mu^{N,m}_T\rangle=\langle \phi, \mu^{N,m}_0\rangle+\int_0^T\frac{1}{N}\sum_{i\in\cN(t)}v_i^N(t)\cdot\nabla_x\phi(x_i^N(t),v_i^N(t))1_{\{m_i^N(t)=m\}}dt\\
&-c(m)\int_0^T\frac{1}{N}\sum_{i\in\cN(t)}v_i^N(t)\cdot\nabla_v\phi(x_i^N(t),v_i^N(t))1_{\{m_i^N(t)=m\}}dt\\
&+c(m)\int_0^T\frac{1}{N}\sum_{i\in\cN(t)}\nabla_v\phi(x_i^N(t),v_i^N(t))1_{\{m_i^N(t)=m\}}\cdot dB_i(t)\\
&+\mu c(m)^2\int_0^T\frac{1}{N}\sum_{i\in\cN(t)}\Delta_v\phi(x_i^N(t),v_i^N(t))1_{\{m_i^N(t)=m\}}dt\\
&+\frac{1}{2}c(m)^2\int_0^T\frac{1}{N}\sum_{i\in\cN(t)}\cQ(x_i^N(t),x_i^N(t))\Delta_v\phi(x_i^N(t),v_i^N(t))1_{\{m_i^N(t)=m\}}dt\\
&+c(m)\int_0^T\frac{1}{N}\sum_{i\in\cN(t)}\sum_{k\in K}\sigma_k(x_i^N(t))\cdot\nabla_v\phi(x_i^N(t),v_i^N(t))1_{\{m_i^N(t)=m\}} dW_t^k\\
&+\int_0^T\frac{1}{N^2}\sum_{n=1}^{m-1}\sum_{\substack{i,j\in\cN(t), \\ i\neq j}}s(n,m-n)|v^N_i(t)-v^N_j(t)|\theta^\ep(x_i^N(t)-x_j^N(t))\\
&\quad\quad\quad\quad \cdot\phi\left(x_i^N(t),\frac{nv_i^N(t)+(m-n)v^N_j(t)}{m}\right)\frac{n}{m}1_{\{m_i^N(t)=n,\,m_j^N(t)=m-n\}}dt\\
&+\int_0^T\frac{1}{N^2}\sum_{n=1}^{m-1}\sum_{\substack{i,j\in\cN(t), \\ i\neq j}}s(n,m-n)|v^N_i(t)-v^N_j(t)|\theta^\ep(x_i^N(t)-x_j^N(t))\\
&\quad\quad\quad\quad \cdot \phi\left(x_j^N(t),\frac{nv_i^N(t)+(m-n)v^N_j(t)}{m}\right)\frac{m-n}{m}1_{\{m_i^N(t)=n,\, m_j^N(t)=m-n\}}dt\\
&-\int_0^T\frac{2}{N^2}\sum_{n=1}^{M}\sum_{i,j\in\cN(t), \; i\neq j}s(n,m)|v^N_i(t)-v^N_j(t)|\theta^\ep(x_i^N(t)-x_j^N(t))\\
&\quad\quad\quad\quad \cdot\phi(x_i^N(t),v_i^N(t))1_{\{m_i^N(t)=m, \, m_j^N(t)=n\}}dt\\
&+ M^{N,J}_T
\end{align*}
where $\{M^{N,J}_t\}_{t\ge0}$ is a martingale associated with coagulation (or jumps) that we do not write explicitly. We can rewrite, using the simplification \eqref{enhanced-diff}, the previous identity more compactly as
\begin{align}\label{iden-emp}
&\langle \phi, \mu^{N,m}_T\rangle=\langle \phi, \mu^{N,m}_0\rangle+\int_0^T\Big\langle v\cdot\nabla_x\phi-c(m)v\cdot \nabla_v+c(m)^2\kappa\Delta_v, \;\mu^{N,m}_t\Big\rangle dt\nonumber\\
&+c(m)\int_0^T\frac{1}{N}\sum_{i\in\cN(t)}\nabla_v\phi(x_i^N(t),v_i^N(t))1_{\{m_i^N(t)=m\}}\cdot dB_i(t)\nonumber\\
&+c(m)\int_0^T\sum_{k\in K}\Big\langle \sigma_k(x)\cdot\nabla_v\phi,\;\mu_t^{N,m}\Big\rangle dW_t^k + M_T^{N,J}\nonumber\\
&+\int_0^T\sum_{n=1}^{m-1}\frac{n}{m}\Big\langle s(n,m-n)|v-w|\theta^\ep(x-y)\phi\Big(x,\frac{nv+(m-n)w}{m}\Big), \; \mu^{N,n}_t(dx,dv)\mu^{N,m-n}_t(dy,dw)\Big\rangle dt\nonumber\\
&+\int_0^T\sum_{n=1}^{m-1}\frac{m-n}{m}\Big\langle s(n,m-n)|v-w|\theta^\ep(x-y)\phi\Big(y,\frac{nv+(m-n)w}{m}\Big), \; \mu^{N,n}_t(dx,dv)\mu^{N,m-n}_t(dy,dw)\Big\rangle dt\nonumber\\
&-\int_0^T2\sum_{n=1}^{M}\Big\langle s(n,m)|v-w|\theta^\ep(x-y)\phi(x,v), \; \mu^{N,m}_t(dx,dv)\mu^{N,n}_t(dy,dw)\Big\rangle dt.
\end{align}
We expect that $M_T^{N,J}$ and the stochastic integrals in $dB_i(t)$ vanish in limit as $N\to\infty$ in $L^2(\P)$, whereas the martingale associated with the common noise persists in the limit. Further, suppose that we have proved that the laws of the collection of $\cD_T(\cM_{1,+})^M$-valued random variables $\{\mu^{N,m}_t:t\in[0,T]\}_{m=1}^M$, $N\in\N$, is tight hence weakly relatively compact. Consider any weak subsequential limit
\begin{align}\label{subseq-conv}
\{\mu^{N_\ell,m}_t:t\in[0,T]\}_{m=1}^M\overset{\ell\to\infty}{\to} \{\ovl{\mu}^m_t:t\in[0,T]\}_{m=1}^M.
\end{align}
For the sake of arguments, apply Skorohod's representation theorem and there exists some auxiliary probability space and on which a sequence of random variables having the same laws as the ones in \eqref{subseq-conv} so that the above convergence holds almost surely. By an abuse of notation, below we use the same letters for the variables on the auxiliary space. Assume that we can prove that $\ovl{\mu}^m_t$ has a density $f_m(t,x,v)$ with respect to Lebesgue measure for every $m$ and $t$. Then, with minor work the linear part of the identity \eqref{iden-emp} converges as $\ell\to\infty$, i.e. 
\begin{align*}
&\langle \phi, \mu^{N_\ell,m}_T\rangle \to \langle \phi, f_m(T,x,v)\rangle,\\
&\langle \phi, \mu^{N_\ell,m}_0\rangle \to \langle \phi,f^0_m(x,v)\rangle, \\
&\int_0^T\left\langle v\cdot\nabla_x\phi-c(m)v\cdot \nabla_v+\kappa c(m)^2\Delta_v, \;\mu^{N_\ell,m}_t\right\rangle dt \\
&\to \int_0^T\left\langle v\cdot\nabla_x\phi-c(m)v\cdot \nabla_v+\kappa c(m)^2\Delta_v, \;f_m(t,x,v)\right\rangle dt,\\
&\int_0^T\sum_{k\in K}\left\langle \sigma_k(x)\cdot\nabla_v\phi,\;\mu_t^{N_\ell,m}\right\rangle dW_t^k \to \int_0^T\sum_{k\in K}\left\langle \sigma_k(x)\cdot\nabla_v\phi,\;f_m(t,x,v)\right\rangle dW_t^k.
\end{align*}
The proof that the nonlinear terms also converge to the corresponding limits is more difficult, hence here we content ourselves with a very heuristic ``two-step argument''. Consider each summand of the last term of \eqref{iden-emp} :
\begin{align*}
\int_0^T\left\langle s(n,m)|v-w|\theta^\ep(x-y)\phi(x,v), \; \mu^{N_\ell,m}_t(dx,dv)\mu^{N_\ell,n}_t(dy,dw)\right\rangle dt, \quad 1\le m,n\le M.
\end{align*}
Assume we take $\ell\to\infty$ first, keeping $\ep$ fixed, we get 
\begin{align*}
&\int_0^T\left\langle s(n,m) |v-w|\theta^\ep(x-y)\phi(x,v), \; \mu^{N_\ell,m}_t(dx,dv)\mu^{N_\ell,n}_t(dy,dw)\right\rangle dt\\
&\overset{\ell\to\infty}{\to} \int_0^T\int s(n,m)|v-w|\theta^\ep(x-y)\phi(x,v) f_m(t,x,v)f_n(t,y,v)dxdydvdwdt\; ;
\end{align*}
then we take $\ep\to0$, and since $\theta^\ep(\cdot)$ approximates the delta-Dirac $\delta_0$, we get 
\begin{align*}
&\int_0^T\int s(n,m)|v-w|\theta^\ep(x-y)\phi(x,v) f_m(t,x,v)f_n(t,y,v)dxdydvdwdt\\
&\overset{\ep\to0}{\to}\int_0^T\int s(n,m)|v-w|\phi(x,v)f_m(t,x,v)f_n(t,x,w)dxdvdwdt.
\end{align*}
Similarly, each summand of the second and third terms from the bottom in \eqref{iden-emp} also converge under the two-step argument (using also $n/m+(m-n)/m=1$)
\begin{align*}
&\int_0^T\frac{n}{m}\left\langle s(n,m-n)|v-w|\theta^\ep(x-y)\phi\left(x,\frac{nv+(m-n)w}{m}\right), \; \mu^{N_\ell,n}_t(dx,dv)\mu^{N_\ell,m-n}_t(dy,dw)\right\rangle dt\\
&+\int_0^T\frac{m-n}{m}\left\langle s(n,m-n)|v-w|\theta^\ep(x-y)\phi\left(y,\frac{nv+(m-n)w}{m}\right), \; \mu^{N_\ell,n}_t(dx,dv)\mu^{N_\ell,m-n}_t(dy,dw)\right\rangle dt\\
&\to\int_0^T\int s(n,m-n)|v-w|\phi\left(x,\frac{nv+(m-n)w}{m}\right)f_n(t,x,v)f_m(t,x,w)dvdwdt.
\end{align*}
Hence, we have (at least under Skorohod's representation) the limit identity satisfied by $\{f_m(t,x,v)\}$:
\begin{align*}
\langle \phi, f_m(T)\rangle =& \langle \phi, f_m^0\rangle +\int_0^T\left\langle v\cdot\nabla_x\phi-c(m)v\cdot \nabla_v+\kappa c(m)^2\Delta_v, \;f_m(t,x,v)\right\rangle dt\\
&+\sum_{n=1}^{m-1}\int_0^T\int s(n,m-n)|v-w|\phi\left(x,\frac{nv+(m-n)w}{m}\right) f_n(t,x,v)f_{m-n}(t,x,w)dxdvdwdt\\
&-2\sum_{n=1}^M\int_0^T\int s(n,m)|v-w|\phi(x,v) f_m(t,x,v)f_{n}(t,x,w)dxdvdwdt\\
&+c(m)\int_0^T\sum_{k\in K}\left\langle \sigma_k(x)\cdot\nabla_v\phi,\;f_m(t,x,v)\right\rangle dW_t^k, \quad m=1,...,M,
\end{align*}
which is the weak formulation of the \abbr{SPDE} system
\begin{align}
\begin{cases}\label{smol-spde}
df_m(t,x,v)=&\left(-v\cdot\nabla_x+c(m)\text{div}_v\left(v\cdot\right)+\kappa c(m)^2\Delta_v\right)f_m(t,x,v)dt \\[5pt]
&-c(m)\sum_{k\in K}\sigma_k(x)\cdot\nabla_vf_m(t,x,v)dW_t^k\\[5pt]
&+\sum_{n=1}^{m-1}\int_{\{nw'+(m-n)w=mv\}}s(n,m-n)|w'-w|f_n(t,x,w')f_{m-n}(t,x,w)dwdw'dt \\[5pt]
&-2\sum_{n=1}^M\int s(n,m)|v-w|f_m(t,x,v)f_n(t,x,w)dwdt \\[10pt]
f_m(\cdot,x,v)|_{t=0}=&f^0_m, \quad m=1,...,M.
\end{cases}
\end{align}
Rewriting the It\^o integral as a Stratonovich integral plus a corrector, we equivalently have that
\begin{align*}
df_m(t,x,v)=&\left(-v\cdot\nabla_x+c(m)\text{div}_v\left(v\cdot\right)+\mu c(m)^2\Delta_v\right)f_m(t,x,v)dt\\
&-c(m)\sum_{k\in K}\sigma_k(x)\cdot\nabla_vf_m(t,x,v)\,\circ\,dW_t^k\\
&+\sum_{n=1}^{m-1}\int_{\{nw'+(m-n)w=mv\}}s(n,m-n)|w'-w|f_n(t,x,w')f_{m-n}(t,x,w)dwdw'dt\\
&-2\sum_{n=1}^M\int s(n,m)|v-w|f_m(t,x,v)f_n(t,x,w)dwdt.
\end{align*}
Thus, we see that the same Stratonovich transport-type noise that acts on the particle system \eqref{particle-sys} also acts on the \abbr{SPDE}.

Lastly, we need to prove that the solution of this \abbr{SPDE} system \eqref{smol-spde} is pathwise unique, which allows to conclude (via nontrivial arguments) that the full sequence of empirical measure converges.
\end{appendix}
\printbibliography 

\end{document}